\documentclass[11pt, reqno]{amsart}
\usepackage{amsfonts}
\usepackage[T1]{fontenc}
\usepackage[dvipsnames, svgnames, table, x11names]{xcolor}
\usepackage{amsmath,amssymb,amsthm,pstricks,amscd,latexsym}
\usepackage[a4paper,centering]{geometry}
\usepackage{upgreek}
\usepackage{float}
\usepackage[pagebackref=true, colorlinks, linkcolor=myblue, anchorcolor=orange, citecolor=myred, urlcolor=Emerald, bookmarksopen, bookmarksdepth=2]{hyperref}
\usepackage[english]{babel}
\usepackage[protrusion=true, expansion=true]{microtype}
\usepackage[perpage]{footmisc}
\usepackage{stackrel}
\usepackage[pdftex]{graphicx}
\usepackage{mathtools}
\usepackage{bookmark}
\usepackage{verbatim}
\usepackage{url}
\usepackage{cleveref}
\usepackage{enumitem}
\usepackage{mathabx}
\usepackage{multirow}
\usepackage[all]{xy}
\usepackage[textsize=footnotesize,backgroundcolor=white]{todonotes}

\usepackage{pgfplots}
\pgfplotsset{compat=newest,compat/show suggested version=false}
\usepackage{tikz,tikz-cd,tkz-euclide}
\usetikzlibrary{fadings,patterns}
\usetikzlibrary{decorations}
\usetikzlibrary{shapes.geometric, positioning}
\usetikzlibrary{arrows, decorations.pathmorphing, decorations.pathreplacing, arrows.meta, bending, decorations.markings}
\usetikzlibrary{automata, positioning, arrows}
\usepackage{tasks}
\definecolor{emerald}{rgb}{0.31, 0.78, 0.47}
\definecolor{royalblue(web)}{rgb}{0.25, 0.41, 0.88}
\definecolor{folly}{rgb}{1.0, 0.0, 0.31}

\definecolor{myblue}{rgb}{0.10, 0.20, 1.90}
\definecolor{myred}{rgb}{1.0, 0.0, 0.25}
\definecolor{mygreen}{rgb}{0.20,0.65, 0.20}

\numberwithin{equation}{section}
\newtheorem{theorem}{Theorem}[section]
\newtheorem{lemma}[theorem]{Lemma}
\newtheorem{proposition}[theorem]{Proposition}
\newtheorem{corollary}[theorem]{Corollary}

\theoremstyle{definition}
\newtheorem{definition}[theorem]{Definition}
\newtheorem{conjecture}[theorem]{Conjecture}

\newtheorem{remark}[theorem]{Remark}
\newtheorem{example}[theorem]{Example}
\newtheorem{notation}[theorem]{Notation}

\newcommand{\thistheoremname}{}
\newtheorem{genericthm}[theorem]{\thistheoremname}
\newenvironment{namedthm}[1]
{\renewcommand{\thistheoremname}{#1}%
	\begin{genericthm}}
	{\end{genericthm}}

\newcommand{\g}{\operatorname{g}}

\renewcommand{\mod}{\operatorname{mod}}
\newcommand{\proj}{\operatorname{proj}}

\newcommand{\TF}{\operatorname{TF}}

\newcommand{\ind}{\operatorname{ind}}

\newcommand{\Hom}{\operatorname{Hom}}

\newcommand{\taurigidpair}{\operatorname{\tau-rigid-pair}}

\renewcommand{\dim}{\operatorname{dim}}

\newcommand{\Ker}{\operatorname{Ker}}

\newcommand{\Fac}{\operatorname{Fac}}

\newcommand{\ftorspair}{\operatorname{f-tors-pair}}

\newcommand{\Sub}{\operatorname{Sub}}
\newcommand{\add}{\operatorname{add}}
\newcommand{\Coker}{\operatorname{Coker}}

\renewcommand{\add}{\operatorname{add}}

\newcommand{\tame}{\operatorname{tame}}

\newcommand{\Cone}{\operatorname{Cone}}

\newcommand{\repr}[1]{%
	{%
		\tiny%
		\begin{pmatrix}%
			#1%
		\end{pmatrix}%
	}%
}

\newcommand\scalemath[2]{\scalebox{#1}{\mbox{\ensuremath{\displaystyle #2}}}}

\begin{document}
	\title[]{The cones of $\textbf{g}$-vectors}
	\author{Mohamad Haerizadeh}
	\email{hyrizadeh@gmail.com}
	\author{Siamak Yassemi}
	\email{syassemi@purdue.edu}
	
	\date{\today}
	\begin{abstract}
		This paper studies the wall-chamber structures of finite-dimensional ($\tau$-tilting infinite) algebras via generic decompositions of g-vectors. In particular, we examine regions outside the chambers. We show that the cones of g-vectors are rational and simplicial. Moreover, we prove that the open cone of a given g-vector coincides with the interior of its $\TF$-equivalence class if and only if the two have the same dimension. Furthermore, we establish that g-vectors satisfy the ray condition when they are sufficiently far from the origin. As an application, we generalize several results of Asai and Iyama concerning $\TF$-equivalence classes of g-vectors.
	\end{abstract}
	\maketitle
	\tableofcontents
	\section{Introduction}
	Derksen and Fei \cite{DeFe15} introduced generic decompositions of (projective) presentations to generalize Kac and Schofield's results \cite{Ka82, Sch92} on the decompositions of quiver representations (see also \cite{CBSc02, DW17}).
	They constructed a notion of decomposition for g-vectors, named ``generic decomposition'' obtained from decompositions of general presentations. More precisely, if $g=g_1\oplus g_2\oplus\ldots\oplus g_s$ is the generic decomposition of $g$, then a general presentation in $\Hom_{\Lambda}(g)$ can be expressed as a direct sum of general presentations in $\Hom_{\Lambda}(g_i)$, for $1\le i\le s$.
	This approach has also been applied to study the categorification of cluster algebras \cite{DK08, Pla13, Fe25}, g-vector fans \cite{PYK23, Yu23, AY23, HY25}, and wall-chamber structures \cite{As22, As21, AsIy24}.
	Moreover, in our paper \cite{HY25b}, we established that there is a natural correspondence between decompositions of generically $\tau$-regular components of representation varieties, in the sense of \cite{CBSc02}, and generic decompositions of g-vectors. This led us to provide a partial response to \cite[Conjecture 6.3]{CILFS14}. We also refer readers to \cite{GLFS21, GLFS24, GLFW23, Ge23, HY25c}  for recent significant works on generically $\tau$-regular components.
	
	On the other hand, inspired by the works of King \cite{Ki94} and Bridgeland \cite{Br17}, Brüstle, Smith, and Treffinger \cite{BST19} studied the wall-chamber structures of algebras. They showed that the open cone associated with any $2$-term silting complex must be a chamber. Subsequently, using $\TF$-equivalence classes, Asai \cite{As21} established the converse. More specifically, he proved that for a finite-dimensional algebra, the following sets are equivalent:
	\begin{itemize}
		\item The set of all chambers of the wall-chamber structure of $\Lambda$
		\item The set of all $\TF$-equivalence classes of dimension $|\Lambda|$
		\item The set of all open cones associated with $2$-term silting complexes
	\end{itemize}
	Asai further demonstrated that the g-vector fan of $\Lambda$ covers the entire ambient space if and only if $\Lambda$ is $\tau$-tilting finite\footnote{The ``if'' part proved in the paper \cite{DIJ19}.}, meaning it admits only finitely many $\tau$-tilting modules.
	
	This naturally raises the problem of understanding the regions outside the chambers of the wall-chamber structure of ($\tau$-tilting infinite) finite-dimensional algebras. In this context, Asai and Iyama \cite{AsIy24} investigated the relationship between $\TF$-equivalence classes of g-vectors and their generic decompositions. They showed that for a finite-dimensional algebra $\Lambda$, and g-vectors $g$ and $h$, the condition $\ind(g)=\ind(h)$ implies that $[g]_{\TF}=[h]_{\TF}$. Furthermore, if $\Lambda$ is $E$-tame or hereditary, then
	$[g]_{\TF}=\Cone^{\circ}\{\ind(g)\}$.
	But in general, if $g$ does not satisfy the ray condition (for example, see \cite[Example 5.9]{AsIy24}), then $\Cone^{\circ}\{\ind(g)\}\neq\Cone^{\circ}\{\ind(tg)\}$, for some $t\in\mathbb{N}$.
	Nevertheless, Asai and Iyama \cite[Conjecture 1.2]{AsIy24} expect that $[g]_{\TF}=\Cone^{\circ}\{\ind(\mathbb{N}g)\}$.
	
	For this reason, we focus on $\Cone\{\ind(\mathbb{N}g)\}$ which is called the \textit{cone of} $g$.
	Due to technical reasons (for instance, see Theorem \ref{962650431975}), we have modified the conjecture mentioned in the paragraph above, and proposed the following slightly modified version.
	\begin{conjecture}\label{362256186221}
		Let $g$ be a g-vector. Then, 
		$[g]_{\TF}^{\circ}= \Cone^{\circ}\{\ind(\mathbb{N}g)\}$.
	\end{conjecture}
	The main purpose of this paper is to study the cones of g-vectors and provide answers to the above conjecture. As the first result,
	we confirm Conjecture \ref{362256186221} for tame g-vectors in any finite-dimensional algebra. 
	\begin{theorem}[\ref{376103330122}]\label{376103330121}
		Let $g$ be a tame g-vector. Then
		\[[g]_{\TF}=[g]^{\circ}_{\TF}=\Cone^{\circ}\{\ind(g)\}.\]
	\end{theorem}
	Secondly, for an arbitrary g-vector $g$, we establish that the cone of $g$ is a simplicial rational polyhedral convex cone. This allows us to prove Proposition \ref{061272712412}, which provides a necessary condition for Conjecture \ref{362256186221}.
	\begin{theorem}[\ref{076273627938} and \ref{355706191811}]
		Let $g$ be a g-vector. Then, there exists $t\in\mathbb{N}$ such that $tg$ satisfies the ray condition (refer to \cite[Definition 2.20]{HY25}), and
		\[\Cone\{\ind(\mathbb{N}g)\}=\Cone\{\ind(tg)\}.\] 
	\end{theorem}
	Thirdly, we establish the following theorem, which is closely related to \cite[Proposition 5.4]{AsIy24}, but does not require the ray condition.
	\begin{theorem}[\ref{962650431975} and \ref{401237755092}]\label{961838105897}
		For a g-vector $g$, consider the following conditions.
		\begin{tasks}[style=itemize](2)
			\task[$(1)$]
			$\scalemath{0.9}{\dim_{\mathbb{R}}\langle [g]_{\TF} \rangle_{\mathbb{R}}= \dim_{\mathbb{R}} \langle\ind(\mathbb{N}g)\rangle_{\mathbb{R}}}$.
			\task[$(2)$]
			$\scalemath{0.9}{[g]_{\TF}^{\circ}=\Cone^{\circ}\{\ind(\mathbb{N}g)\}}$.
			\task[$(3)$] $\scalemath{0.9}{\dim_{\mathbb{R}}W_g= |\Lambda|-|\ind(\mathbb{N}g)|}$.
			\task[$(4)$] $\scalemath{0.9}{W_g=\Ker\langle-,[g]_{\TF}\rangle}$.
		\end{tasks}
		Then,
		\[(1)\Longleftrightarrow(2)\Longleftarrow(2)+(4)\Longleftrightarrow(3).\]
	\end{theorem}
	Inspired by the proof of \cite[Proposition 5.4]{AsIy24}, we prove the following lemma, which does not require the ray condition and is crucial for demonstrating Theorem \ref{961838105897}.
	\begin{lemma}[\ref{988807280630}]
		Let $g$ be a g-vector. Then
		\[\Cone\{\ind(\mathbb{N}g)\}\cap [g]^{\circ}_{\TF}\subseteq\Cone^{\circ}\{\ind(\mathbb{N}g)\}\subseteq [g]_{\TF}.\]
		Especially, if $g\in[g]^{\circ}_{\TF}$, then
		\[\Cone\{\ind(\mathbb{N}g)\}\cap [g]^{\circ}_{\TF}=\Cone^{\circ}\{\ind(\mathbb{N}g)\}.\]
	\end{lemma}
	Furthermore, we develop a tool that allows us to prove that each subset of $\TF^{\mathrm{ss}}_{\mathbb{Z}}(\Lambda)$ has a minimal element (see Corollary \ref{868524213012}). This tool also plays a crucial role in the proof of Theorem \ref{948850386317}.
	\begin{proposition}[\ref{206559753624}]
		Let $\theta,\eta\in K_0(\proj\Lambda)_{\mathbb{R}}$. If $\eta\in\partial [\theta]_{\TF}\setminus [\theta]_{\TF}$, then
		\[\dim_{\mathbb{R}}\langle [\eta]_{\TF} \rangle_{\mathbb{R}}\lneqq\dim_{\mathbb{R}}\langle [\theta]_{\TF} \rangle_{\mathbb{R}}.\]
	\end{proposition}
	Subsequently, we break down Conjecture \ref{362256186221} into three simpler conjectures designed to provide more accessible paths for future investigations. Below, we provide evidence supporting this conjecture.
	\begin{conjecture}\label{404319523362}
		Let $g$ be a g-vector. Then, the following conditions hold:
		\begin{itemize}
			\item[$(1)$] Rational points are dense in $[g]_{\TF}$.
			\item[$(2)$] $\dim_{\mathbb{R}}\langle\ind(\mathbb{N}g)\rangle_{\mathbb{R}}=1$ if and only if $[g]_{\TF}=\mathbb{R}^{> 0}g$.
			\item[$(3)$] $|\TF^{\mathrm{ss}}_{\mathbb{Z}}(g)|<\infty$ (see Definition \ref{142707043868}).
		\end{itemize}
	\end{conjecture}
	The first condition is equivalent to the existence of a convex rational polyhedral cone within $[g]_{\TF}$ of dimension $\dim_{\mathbb{R}}\langle[g]_{\TF}\rangle_{\mathbb{R}}$.
	So instead of $\ind(\mathbb{N}g)$ (see Theorem \ref{961838105897}), it is enough to find arbitrary g-vectors $\{h^{1},h^{2},\ldots,h^{s}\}$ in $[g]_{\TF}$ such that
	\[\dim_{\mathbb{R}}\langle h^{i}\mid 1\le i\le s\rangle_{\mathbb{R}}=\dim_{\mathbb{R}}\langle [g]_{\TF}\rangle_{\mathbb{R}}.\]
	Moreover, the second condition states that Conjecture \ref{362256186221} holds for a g-vector $g$ with $\dim_{\mathbb{R}}\langle\ind(\mathbb{N}g)\rangle_{\mathbb{R}}=1$.
	Furthermore, the third condition means that there are only finitely many $\TF$-equivalence classes of g-vectors contained in $\overline{[g]_{\TF}}$. We expect the first two conditions to follow from the third.
	\begin{theorem}[\ref{948850386317}]
		Conjecture \ref{362256186221} and Conjecture \ref{404319523362} are equivalent.
	\end{theorem}
	According to \cite{As21}, we know that the condition $\dim_{\mathbb{R}}\langle\ind(\mathbb{N}g)\rangle_{\mathbb{R}}=|\Lambda|$ implies that $g$ belongs to the interior of the g-vector fan. Therefore, in this case, $g$ is a tame g-vector and by Theorem \ref{376103330121}, $[g]_{\TF}=\Cone^{\circ}\{\ind(g)\}$. Moreover, we establish the conjecture for cones of dimension $|\Lambda|-1$.
	\begin{corollary}[\ref{401237755092}]
		Let $g$ be a g-vector. If $\dim_{\mathbb{R}}\langle\ind(\mathbb{N}g)\rangle_{\mathbb{R}}=|\Lambda|-1$, then $[g]_{\TF}^{\circ}=\Cone^{\circ}\{\ind(\mathbb{N}g)\}$.
	\end{corollary}
	\section{Conventions}
	Throughout this paper, let $\Lambda$ a basic finite-dimensional algebra over an algebraically closed field $k$. By Morita equivalence, we may assume $\Lambda=kQ/I$, where $(Q, I)$ is a bound quiver with $n$ vertices. To simplify notation, we assume that $Q$ is connected. In this setting, there are $n$ simple modules up to isomorphism, typically denoted by $\{S_{(1)},S_{(2)},\ldots,S_{(n)}\}$. The Grothendieck group of the category of finitely generated $\Lambda$-modules, denoted by $K_{0}(\mod\Lambda)$, is a free abelian group of rank $n$ with a basis $\{[S_{(1)}],[S_{(2)}],\ldots,[S_{(n)}]\}$. Similarly, the Grothendieck group of the category of finitely generated projective $\Lambda$-modules, $K_0(\proj\Lambda)$, is also a free abelian group of rank $n$, with a basis consisting of the classes of projective modules $\{[P_{(1)}],[P_{(2)}],\ldots,[P_{(n)}]\}$, where $P_{(i)}$ is the projective cover of the simple module $S_{(i)}$. The elements of $K_0(\mod\Lambda)_\mathbb{R}$ and $K_0(\proj\Lambda)_\mathbb{R}$  are viewed as vectors in $\mathbb{R}^n$. Finally, we fix a complete set of primitive orthogonal idempotents $\{e_1,\ldots,e_n\}$ corresponding to $\{P_{(1)},\ldots, P_{(n)}\}$.
	\section{Preliminaries and background}
	An approach to studying the wall-chamber structures of finite-dimensional algebras is to explore $\TF$-equivalence classes.
	In this point of view, the stability spaces and the cones associated with $\tau$-rigid pairs play central roles. 
	This section reviews key findings from \cite{As21, AsIy24} regarding the relationships of these notions. Subsequently, the correspondence among the chambers and the cones of $\tau$-tilting pairs is also studied. This is what demonstrates the relationship between the wall-chamber structure and the g-vector fan \cite{As21}. Although most of these results are not used directly in the proofs of our main results, they are essential for understanding the themes of this paper.
	
	Let $\mathcal{C}$ be a full subcategory of $\mod\Lambda$.
	\begin{itemize}
		\item $\mathcal{C}$ is called \textit{covariantly finite} if every $\Lambda$-module $X$ has a left $\mathcal{C}$-approximation, that is, there is a map $X\rightarrow C$ with $C\in\mathcal{C}$ such that for all $C'\in\mathcal{C}$ the induced map $\Hom(C,C')\rightarrow\Hom(X,C')$ is onto.
		\item $\mathcal{C}$ is called \textit{contravariantly finite} if every $\Lambda$-module $X$ has a right $\mathcal{C}$-approximation, that is, there is a map $C\rightarrow X$ with $C\in\mathcal{C}$ such that for all $C'\in\mathcal{C}$ the induced map $\Hom(C',C)\rightarrow\Hom(C',X)$ is onto.
			\item $\mathcal{C}$ is called \textit{functorially finite} if it is both covariantly finite and contravariantly finite.
	\end{itemize}
		A full subcategory of $\mod\Lambda$ is said to be a \textit{torsion class} (respectively, a \textit{torsion-free class}) if it is closed under quotients (respectively, sub-modules) and extensions. A pair $(\mathcal{T},\mathcal{F})$ forms a \textit{torsion pair} if any of the following equivalent statements hold:
		\begin{itemize}
			\item $\mathcal{T}=\prescript{\perp}{}{\mathcal{F}}$ \footnote{$\prescript{\perp}{}{\mathcal{F}}:=\{X\in\mod\Lambda\mid\Hom_{\Lambda}(X,\mathcal{F})=0\}$}, and $\mathcal{F}=\mathcal{T}^{\perp}$ \footnote{$\mathcal{T}^{\perp}:=
				\{Y\in\mod\Lambda\mid\Hom_{\Lambda}(\mathcal{T},Y)=0\}$}.
			\item $\mathcal{T}$ is a torsion class, and $\mathcal{F}=\mathcal{T}^{\perp}$.
			\item $\mathcal{F}$ is a torsion-free class, and $\mathcal{T}=\prescript{\perp}{}{\mathcal{F}}$.
			\item $\Hom_{\Lambda}(\mathcal{T},\mathcal{F})=0$ and for each $X\in\mod\Lambda$, there exists an exact sequence of the form:
				\[\begin{tikzcd}[cramped]
					0 & T & X & F & {0,}
					\arrow[from=1-1, to=1-2]
					\arrow["a", from=1-2, to=1-3]
					\arrow["b", from=1-3, to=1-4]
					\arrow[from=1-4, to=1-5]
				\end{tikzcd}\]
			where $T\in\mathcal{T}$ and $F\in\mathcal{F}$.
		\end{itemize}
		
		In the above exact sequence, the morphism $a$ is a right $\mathcal{T}$-approximation, and $b$ is a left $\mathcal{F}$-approximation. Thus, torsion classes are contravariantly finite and torsion-free classes are covariantly finite.
		
		Let $X$ be a finite-dimensional module.
		Denote $\Fac(X)$ (respectively, $\Sub(X)$) as the smallest full subcategory of $\mod(\Lambda)$ that contains $X$, closed under taking quotients (respectively, closed under submodules)  and extensions.
		According to \cite[Proposition 1.1 and Theorem 2.7]{AIR14},
		for a torsion pair $(\mathcal{T},\mathcal{F})$ of $\mod\Lambda$, the following statements are equivalent:
		\begin{itemize}
			\item $\mathcal{T}$ is functorially finite.
			\item $\mathcal{F}$ is functorially finite.
			\item $\mathcal{T}=\Fac M$, for some $\tau$-rigid module $M$.
			\item $\mathcal{F}=\Sub(\tau M\oplus \nu P)$, for some $\tau$-tilting pair $(M,P)$.
		\end{itemize}
		Therefore, the map
		\[\begin{array}{rcl}
			\taurigidpair\Lambda &\longrightarrow &\ftorspair\Lambda, \\
			(M, P)& \longmapsto &(\Fac M, \Sub(\tau M\oplus \nu P))
		\end{array}\footnote{$\nu$ is the Nakayama functor.}\]
		provides a bijection between $\tau$-tilting pairs and functorially finite torsion pairs.
		
	In the following, we review the relationship among g-vector fan and the wall-chamber structure.  To do so, we recall some basic facts regarding stability conditions and the cones associated with $\tau$-rigid pairs. 
	
	\begin{namedthm}{Definition-Theorem}[{\cite[Subsection 3.1]{BKT14}}]
		For $\theta\in\mathbb{R}^n$, define
		\[\begin{array}{l}
			\overline{\mathcal{T}_{\theta}}:=\{M\in\mod\Lambda\mid \langle\theta,[M']\rangle \ge 0, \text{for all factor module $M'$ of $M$}\}, \\
			\overline{\mathcal{F}}_{\theta}:=\{M\in \mod\Lambda\mid\langle\theta,[M']\rangle \le 0, \text{for all sub-module $M'$ of $M$}\}, \\
			\mathcal{T}_{\theta}:=\{M\in \mod\Lambda\mid \langle\theta,[M']\rangle > 0, \text{for all non-zero factor module $M'$ of $M$}\}, \\
			\mathcal{F}_{\theta}:=\{M\in \mod\Lambda\mid\langle\theta,[M']\rangle < 0, \text{for all non-zero sub-module $M'$ of $M$}\},
		\end{array}\]
		\[\begin{array}{l l}
			\mathcal{W}_{\theta}:= \overline{\mathcal{T}}_{\theta}\bigcap \overline{\mathcal{F}}_{\theta}, & W_\theta:=\langle\{[X]\mid X\in\mathcal{W}_{\theta}\}\rangle_{\mathbb{R}}.
		\end{array}\]
		Then, the pairs $(\overline{\mathcal{T}}_{\theta}, \mathcal{F}_{\theta})$ and $(\mathcal{T}_{\theta}, \overline{\mathcal{F}}_{\theta})$ are torsion pairs which
		are called \textit{semistable torsion pairs}, and $\mathcal{W}_{\theta}$ is a wide subcategory which is said to be \textit{$\theta$-semistable subcategory}.
	\end{namedthm}
	
	The \textit{stability space} for $X\in\mod(\Lambda)$ is
	\[\mathfrak{D}(X):=\{\theta\in\mathbb{R}^n\mid X\in\mathcal{W}_{\theta}\}.\]
	Any stability space of codimension $1$ is called a \textit{wall}, and any connected component of the space
	\[\mathfrak{R}=\mathbb{R}^n\setminus\overline{\bigcup_{X\in\mod(\Lambda)}\mathfrak{D}(X)}\]
	is said to be \textit{chamber}.
	\begin{definition}\label{772608708051}
		Consider a subset $\mathcal{X}$ of $\mathbb{R}^{n}$. We say it is a \textit{convex cone}, if for all $x,y\in\mathcal{X}$ and $u,v\in\mathbb{R}^{\ge 0}$, $ux+vy$ belongs to $\mathcal{X}$. Moreover, $\mathcal{X}$ is called \textit{rational cone} (respectively, \textit{polyhedral cone}, \textit{simplicial cone}), if there exist rational vectors  (respectively, finitely many vectors, linearly independent vectors) $\{x_i\}_{i\in I}$ such that
		\[\mathcal{X}=\Cone\{x_i\mid i\in I\}:=\sum_{i\in I}\mathbb{R}^{\ge 0}x_i.\]
	\end{definition}
	An element of $K_0(\proj\Lambda)$ is said to be \textit{g-vector}.
	For a complex
	\[\mathcal{P}=\ldots\rightarrow P^{-1}\rightarrow P^0\rightarrow P^1\rightarrow\ldots,\]
	in $\in K^b(\proj\Lambda)$, the g-vector of $\mathcal{P}$ is defined as
	\[g^{\mathcal{P}}:=\sum_{i\in\mathbb{Z}}(-1)^i[P^i]\in K_0(\proj\Lambda)\cong\mathbb{Z}^n.\]
	Moreover, for the minimal projective presentation $a:P^{-1}\rightarrow P^0$ of a module $M$, the g-vector of $M$ is defined as
	\[g^M:=g^{\mathcal{P}_a}\in K_0(\proj\Lambda),\]
	where $\mathcal{P}_a$ is the corresponding $2$-term complex in $K^b(\proj\Lambda)$. In this case, the g-vector of the complex $\mathcal{P}_a$ is denoted by $g^a$.
	
	The \textit{g-vector fan} of $\Lambda$ is the union of open cones
	\[\mathfrak{C}^{\circ}_{(M,P)}:=\Cone^{\circ}\{g^{M_1},\ldots,g^{M_s},-g^{P_1},\ldots,-g^{P_r}\},\]
	where $(M,P)$ is a $\tau$-rigid pair, and $M=M_1\oplus\ldots\oplus M_s$ and $P=P_1\oplus\ldots\oplus P_r$ are the Krull-Schmidt decompositions of $M$ and $P$, respectively\footnote{Since there is an appropriate bijection between $\tau$-tilting pairs and $2$-term silting complexes, one can define the notion of g-vector fan by considering g-vectors of $2$-term silting complexes.}.
	
	The following result is a direct consequence of \cite[Proposition 3.19]{As21} and \cite[Theorem 2.27]{AiIy12}. It was used by Asai \cite{As21} to establish the correspondence between the cones associated with $\tau$-tilting pairs and chambers in the wall-chamber structure. Later in this paper, we apply this result to prove Corollary \ref{401237755092}.
	\begin{proposition}\label{129173072599}
		Let $g$ be a g-vector. Then, the following statements are equivalent.
		\begin{itemize}
			\item[$(1)$] There exists a $\tau$-tilting pair $(M,P)$ such that $g \in\mathfrak{C}^{\circ}_{(M,P)}$.
			\item[$(2)$] $\mathcal{W}_g=\{0\}.$
			\item[$(3)$] $g$ lies in a chamber.
			\item[$(4)$] $g$ is not on any wall.
		\end{itemize}
		In this case, $g$ is (tame and) contained in $\langle g^{M_i},-g^{P_j}\mid 1\le i\le s, 1\le j\le r \rangle_{\mathbb{N}}$, where $M=M_1\oplus M_2\oplus\ldots\oplus M_s$ and $P=P_1\oplus P_2\oplus\ldots\oplus P_r$ are Krull-Schmidt decompositions.
	\end{proposition}
	\begin{example}
		Consider the following quiver $Q$.
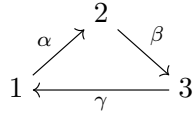
		\[\begin{tikzcd}[cramped,sep=scriptsize,scale=0.4]
			&& 2 && \\
			& 1 && 3
			\arrow["\beta", from=1-3, to=2-4]
			\arrow["\alpha", from=2-2, to=1-3]
			\arrow["\gamma", from=2-4, to=2-2]
		\end{tikzcd}\]
		In Figure \ref{Q3SkQNPExWLC}, we see three chambers and seven walls of the wall-chamber structure of the algebra $kQ/\langle (\alpha\gamma\beta)^{3}\rangle$ where
		\begin{tasks}[style=itemize](3)
			\task {\tiny The blue cone:} $\scalemath{0.8}{\mathfrak{C}^{\circ}_{(\repr{2\\3\\1}\oplus \repr{2\\3}\oplus \repr{2},\repr{0})}}$
			\task {\tiny The red cone:}
			$\scalemath{0.8}{\mathfrak{C}^{\circ}_{(\repr{1\\2\\3}\oplus\repr{2\\3\\1}\oplus\repr{2},\repr{0})}}$
			\task {\tiny The green cone:}
			$\scalemath{0.8}{\mathfrak{C}^{\circ}_{(\repr{2\\3}\oplus \repr{2},\repr{1\\2\\3})}}$.
		\end{tasks}
		By \cite[Example 3.30]{BST19}, this finite-dimensional algebra has only $20$ $\tau$-tilting pairs. Therefore, by \cite[Theorem 4.7]{As21}, its g-vector fan covers the whole of $\mathbb{R}^3$.
		\begin{figure}[htp]
			\[\begin{tikzpicture}[scale=0.6]
			\coordinate (0) at (0,0);
			\coordinate (1) at (0,2);
			\coordinate (2) at (-2,0);
			\coordinate (3) at (-2,2);
			\coordinate (4) at (-0.5,1.2);
			\coordinate (7) at (2,0);
			\coordinate (7') at (8,0);
			\coordinate (2')  at (-8,0);
			\coordinate (3')  at (-6,6);
			\coordinate (1')  at (0,8);
			\coordinate (4')  at (-2.5,6);
			\coordinate (5) at (-8,4);
			\coordinate (6) at (-4,7);
			\coordinate (8) at (5,6);
			\draw (0) -- (7');
			\draw (1') -- (7');
			\draw[dotted] (4') -- (7');
			\draw (0) -- (1');
			\draw (1') -- (3');
			\draw[dotted] (1') -- (4');
			\draw (2') -- (3');
			\draw (0) -- (2');
			\draw (0) -- (3');
			\draw (2') -- (3');
			\draw[dotted] (3') -- (4');
			\draw[dotted] (0) -- (4');
			\draw[dotted] (2') -- (4');
			\filldraw[draw=black,color=folly!50, opacity=0.40]
			(0) -- (1') -- (4');
			\filldraw[draw=black,color=royalblue(web)!100, opacity=0.40]
			(0) -- (1') -- (3');
			\filldraw[draw=black,color=emerald!100, opacity=0.40]
			(0) -- (3') -- (4');
			\filldraw[draw=black,color=emerald!100, opacity=0.20]
			(0) -- (2') -- (3');
			\filldraw[draw=black,color=emerald!100, opacity=0.40]
			(0) -- (2') -- (4');
			\filldraw[draw=black,color=folly!50, opacity=0.40]
			(0) -- (7') -- (4');
			\filldraw[draw=black,color=folly!50, opacity=0.20]
			(0) -- (7') -- (1');
			\node at (0) [below] {$\scalemath{0.5}{(0,0,0)}$};
			\node at (3) [left] {$\scalemath{0.5}{(-1,1,0)}$};
			\node at (2) [above] {$\scalemath{0.5}{(-1,0,0)}$};
			\node at (4) [above] {$\scalemath{0.5}{(0,1,-1)}$};
			\node at (1) [right] {$\scalemath{0.5}{(0,1,0)}$};
			\node at (5) [above] {};
			\node at (6) [above] {};
			\node at (7) [above] {$\scalemath{0.5}{(1,0,0)}$};
			\node at (8) [below] {};
		\end{tikzpicture}\]
		\caption{Three chambers and seven walls of the wall-chamber structure of the algebra $kQ/\langle (\alpha\gamma\beta)^{3}\rangle$}\label{Q3SkQNPExWLC}
		\end{figure}
	\end{example}
	In this paper, we focus on points outside the g-vector fan. Based on the last proposition, we can assert that if a given g-vector does not belong to the g-vector fan, then the associated wide subcategories do not vanish. Therefore, following \cite[Proposition 2.8]{As21}, we can assert that each g-vector outside the g-vector fan lies on a wall. 
	Moreover, the following result reveals that the semistable torsion pairs associated with these g-vectors are no longer functorially finite.
	\begin{proposition}
		Let $\theta\in\mathbb{R}^n$.
		Then, $\theta$ belongs to $\mathfrak{C}^{\circ}_{(M,P)}$,
		for some $\tau$-rigid pair $(M,P)$, if and only if
		\[\begin{array}{rl}
			\scalemath{0.95}{(\mathcal{T}_{\theta}, \overline{\mathcal{F}}_{\theta})=(\Fac M, M^{\perp})}, &\scalemath{0.95}{(\overline{\mathcal{T}}_{\theta}, \mathcal{F}_{\theta})=(\prescript{\perp}{}{\tau M}\cap P^{\perp},\Sub(\tau M\oplus\nu P))}.
		\end{array}\]
		Therefore, the torsion class $\overline{\mathcal{T}}_{\theta}$ is functorially finite if and only if $\theta$ belongs to the $\g$-vector fan of $\Lambda$. 
		In this case, $\mathcal{W}_{\theta}=\prescript{\perp}{}{\tau M}\cap P^{\perp}\cap M^{\perp}$.
	\end{proposition}
	\begin{proof}
		Refer to \cite[Proposition 3.11]{As21} or \cite[Proposition 3.3]{Yu18}.
	\end{proof}
	
	Torsion pairs associated with morphisms in $\proj\Lambda$ play a crucial role in the study of generic decompositions, semistable torsion pairs, and $\TF$-equivalence classes. In the following, we recall this notion and review some foundational results (see, for instance, Proposition \ref{341408431998} and Lemma \ref{076626606506}). Using these tools, we then provide a proof for our result in Theorem \ref{310913388628}.
	\begin{remark}
		For a module $X$, it is easy to check that $\mathsf{T}(X):=\prescript{\perp}{}{(X^{\perp})}$ is the smallest torsion class, and $\mathsf{F}(X):=(\prescript{\perp}{}{X})^{\perp}$ is the smallest torsion-free class containing $X$.
	\end{remark}
	
	\begin{definition}[{\cite[Definition 3.1]{AsIy24}}]
		Let $a$ be a morphism in $\proj\Lambda$. Then, the pairs $(\overline{\mathcal{T}}_{a}, \mathcal{F}_{a})$ and $(\mathcal{T}_{a}, \overline{\mathcal{F}}_{a})$ where
		\[\begin{array}{l l l l}
			\mathcal{T}_a:=\mathsf{T}(\Coker a), & \overline{\mathcal{T}}_a:=\prescript{\perp}{}{\Ker\nu a}, & \mathcal{F}_a:=\mathsf{F}(\Ker\nu a), & \overline{\mathcal{F}}_a:=(\Coker a)^{\perp}
		\end{array}\]
		are torsion pairs.
	\end{definition}
	\begin{proposition}[{\cite[Theorem 4.3]{AsIy24}}]\label{341408431998}
		Let $g$ be a g-vector. Then
		\[\begin{array}{c c c c}
			\overline{\mathcal{T}}_g=\bigcup\limits_{a\in A}\overline{\mathcal{T}}_a, & \overline{\mathcal{F}}_g=\bigcup\limits_{a\in A}\overline{\mathcal{F}}_a, & \mathcal{T}_{g}=\bigcap\limits_{a\in A}\mathcal{T}_a, & \mathcal{F}_g=\bigcap\limits_{a\in A}\mathcal{F}_a,
		\end{array}\]
		where $A$ is the set of every morphism $a\in\proj(\Lambda)$ with $g^a=tg$, for some $t\in\mathbb{N}$.
	\end{proposition}
	\section{\texorpdfstring{$\TF$-equivalence classes}{TF-equivalence classes}}
	First, we review the notion of $\TF$-equivalence relation on $K_0(\proj\Lambda)_{\mathbb{R}}$. Then, basic results concerning $\TF$-equivalence classes are studied. Subsequently, we study the relationship between the boundary and interior of the $\TF$-equivalence class of a given vector, which allows us to prove Proposition \ref{206559753624}.
	\begin{definition}[{\cite[Definition 2.13]{As21}}]
		Consider $\theta, \eta\in K_0(\proj\Lambda)_{\mathbb{R}}$. We say $\theta$ and $\eta$ are $\TF$-\textit{equivalent} provided that $\overline{\mathcal{T}}_{\theta}=\overline{\mathcal{T}}_{\eta}$ and $\overline{\mathcal{F}}_{\theta}=\overline{\mathcal{F}}_{\eta}$. The $\TF$-equivalence class of $\theta$ is denoted by $[\theta]_{\TF}$.
	\end{definition}
	\begin{lemma}[{\cite[Lemma 2.16]{As21}}]\label{532756467772}
		Let $\theta$ and $\eta$ be in $K_0(\proj\Lambda)_{\mathbb{R}}$. Then, $\eta$ belongs to $\overline{[\theta]_{\TF}}$ if and only if $\overline{\mathcal{T}}_{\theta}\subseteq\overline{\mathcal{T}}_{\eta}$ and $\overline{\mathcal{F}}_{\theta}\subseteq\overline{\mathcal{F}}_{\eta}$.
	\end{lemma}
	\begin{proposition}[{\cite[Theorem 2.17]{As21}}]\label{987463920193}
		For $\theta, \eta\in K_0(\proj\Lambda)_{\mathbb{R}}$, the following statements are equivalent.
		\begin{itemize}
			\item[$(1)$] $[\theta]_{\TF}=[\eta]_{\TF}$.
			\item[$(2)$] $\{r\theta+(1-r)\eta\mid r\in[0,1]\} =:[\theta,\eta]\subseteq[\theta]_{\TF}$.
			\item[$(3)$] $\forall\gamma\in[\theta,\eta],  \mathcal{W}_{\eta}=\mathcal{W}_{\theta}.$
		\end{itemize}
	\end{proposition}
	\begin{lemma}[{\cite[Proposition 2.9]{AsIy24}}]\label{284590845114}
		Assume that $\theta,\eta\in K_0(\proj\Lambda)_{\mathbb{R}}$ are $\TF$-equivalent. Then, for $M\in\mathcal{T}_{\theta}$, there exists $t\in\mathbb{N}$ such that $M\in\mathcal{T}_{t\theta-\eta}$. Additionally, the similar assertions hold for
		$\overline{\mathcal{T}}_{\theta}$,
		$\overline{\mathcal{F}}_{\theta}$,
		$\mathcal{F}_{\theta}$
		and
		$\mathcal{W}_{\theta}$.
	\end{lemma}
		Let $\mathcal{X}\subseteq K_0(\proj\Lambda)_{\mathbb{R}}$. The set consisting of all points $\eta\in\mathcal{X}$ that belong to some open subset of $\langle\mathcal{X}\rangle_{\mathbb{R}}$ is said to be the \textit{(relative) interior} of $\mathcal{X}$, and denoted by $\mathcal{X}^{\circ}$. Moreover, the \textit{boundary} of $\mathcal{X}$ is defined as $\partial \mathcal{X}:=\overline{\mathcal{X}}\setminus \mathcal{X}^{\circ}$.
		
	\begin{remark}
		For $0\neq\theta\in K_0(\proj\Lambda)_{\mathbb{R}}$, let $\{\theta_1,\ldots,\theta_s\}\subseteq [\theta]_{\TF}$ be a basis for $\langle [\theta]_{\TF} \rangle_{\mathbb{R}}$. Then, the convexity of $\TF$-equivalence classes implies that
		\[\Cone^{\circ}\{\theta_1,\ldots,\theta_s\}\subseteq [\theta]_{\TF}\]
		is an open subset of $\langle [\theta]_{\TF} \rangle_{\mathbb{R}}$. Therefore, the interior of any (non-zero) $\TF$-equivalence class is not empty.
	\end{remark}
	\begin{lemma}\label{434203923932}
		Let $\mathcal{X}\subseteq K_0(\proj\Lambda)_{\mathbb{R}}$ be a convex cone. If $\eta\in \mathcal{X}^{\circ}$ and $\gamma\in\overline{\mathcal{X}}$, then
		\[\{t\eta+(1-t)\gamma\mid t\in (0,1)\}=:(\eta,\gamma)\subseteq \mathcal{X}^{\circ}.\]
		Moreover, $t\eta-\gamma\in\mathcal{X}^{\circ}$, for sufficiently large $t\in\mathbb{N}$.
	\end{lemma}
	\begin{proof}
		Consider an open ball $\mathcal{B}\subseteq \mathcal{X}$ containing $\eta$. Since $\mathcal{X}$ is convex and $\gamma\in\overline{\mathcal{X}}$, for $\alpha\in (\gamma,\eta)$, there exists $\gamma'\in \mathcal{X}^{\circ}$ near $\gamma$ such that
		\[\alpha\in (\gamma',\mathcal{B}):=\{t\gamma'+(1-t)b\mid b\in\mathcal{B}, t\in (0,1)\}\subseteq \mathcal{X}.\]
		This subset is open in $\langle \mathcal{X}\rangle_{\mathbb{R}}$. Therefore, $\alpha\in\mathcal{X}^{\circ}$.
		
		Since $\mathcal{X}$ is a convex cone and $\gamma,\eta\in\langle \mathcal{X}\rangle_{\mathbb{R}}$, we conclude that $(-\gamma,\eta)\cap\overline{\mathcal{X}}\neq\emptyset$. Thus, the second assertion follows from the first one. 
	\end{proof}
	The following lemma later used to prove Corollary \ref{342694140028} and Lemma \ref{833777223956}, which are crucial tools for establishing our main results.
	\begin{lemma}\label{075288509646}
		Consider a convex cone $\mathcal{X}\subseteq K_0(\proj\Lambda)_{\mathbb{R}}$. If $\eta\in \mathcal{X}^{\circ}$ and $\gamma\in\partial \mathcal{X}$, then for all $t\in\mathbb{R}^{>0}$, we have
		\[(t\gamma-\eta,\gamma)\cap\overline{\mathcal{X}}=\emptyset.\]
	\end{lemma}
	\begin{proof}
		We prove the assertion for $t=2$. The remaining cases can be obtained, similarly.
		Let $\alpha\in (2\gamma-\eta,\gamma)\cap\overline{\mathcal{X}}$ (see Figure \ref{460759178161}). Then  by Lemma \ref{434203923932}, we have
		\[\gamma\in (\alpha,\eta)\subseteq \mathcal{X}^{\circ},\]
		which contradicts the assumption.
		\begin{figure}
			\centering
			\textbf{}\par\medskip
			\[\begin{tikzpicture}[scale=0.05, rotate around z=230] 
				\coordinate (A) at (-10,0); 
				\coordinate (B) at (20,80); 
				\coordinate (C) at (50,0); 
				\coordinate (D) at (15,25);
				\coordinate (E) at (25,10);
				\coordinate (G) at (-20,15);
				
				\coordinate (F) at (intersection of A--D and E--G);
				
				\draw (A) -- (D);
				\draw[thick] (E) -- node[pos=0.8, above]{$\alpha$} node[pos=0.75]{$\scalemath{0.5}{\bullet}$} (G);
				\draw (A) -- (B); 
				\draw (B) -- (C); 
				\draw (C) -- (A); 
				\draw[dotted] (B) -- (D);  
				\draw (C) -- (D); 
				
				\node at (B) [right] {0}; 
				\filldraw[black] (E) circle (20pt); 
				\node at (E) [left] {$\eta$};
				\filldraw[black] (G) circle (20pt);
				\node at (G) [right] {$2\gamma-\eta$};
				\filldraw[black] (F) circle (20pt);
				\node at (F) [left] {$\gamma$};
			\end{tikzpicture}\]
			\caption{An illustration of the proof of Lemma \ref{075288509646}} \label{460759178161}
		\end{figure}
	\end{proof}
	\begin{proposition}\label{206559753624}
		Let $\theta,\eta\in K_0(\proj\Lambda)_{\mathbb{R}}$. If $\eta\in\partial [\theta]_{\TF}\setminus [\theta]_{\TF}$, then
		\[\dim_{\mathbb{R}}\langle [\eta]_{\TF} \rangle_{\mathbb{R}}\lneqq\dim_{\mathbb{R}}\langle [\theta]_{\TF} \rangle_{\mathbb{R}}.\]
	\end{proposition}
	\begin{proof}
		It follows from Lemma \ref{532756467772} that $[\eta]_{\TF}$ is included in $\partial [\theta]_{\TF}\setminus [\theta]_{\TF}$. Thus, $\langle [\eta]_{\TF} \rangle_{\mathbb{R}}\subseteq\langle [\theta]_{\TF} \rangle_{\mathbb{R}}$. Therefore,
		\begin{equation}\label{eq305294877718}
			\dim_{\mathbb{R}}\langle [\eta]_{\TF} \rangle_{\mathbb{R}}=\dim_{\mathbb{R}}\langle [\theta]_{\TF} \rangle_{\mathbb{R}},
		\end{equation}
		implies that
		$\langle [\eta]_{\TF} \rangle_{\mathbb{R}}=\langle [\theta]_{\TF} \rangle_{\mathbb{R}}$.
		Without loss of generality, assume that $\theta\in [\theta]_{\TF}^{\circ}$ and $\eta\in [\eta]_{\TF}^{\circ}$. Consider an open ball $B\subseteq[\eta]_{\TF}$ containing $\eta$. Then, \eqref{eq305294877718} implies that $B\cap (\eta,\theta)\neq\emptyset$. On the other hand, by Lemma \ref{434203923932}, we have $(\eta,\theta)\subseteq[\theta]_{\TF}^{\circ}$, and so $B\cap (\eta, \theta)=\emptyset$. This contradicts the equality \eqref{eq305294877718}.
	\end{proof}
	\begin{definition}\label{142707043868}
		Let $\theta$ be an arbitrary vector. Then, we set
		\[\begin{array}{l l}
			\TF^{\mathrm{ss}}_{\mathbb{Z}}(\theta)= \{(\overline{\mathcal{T}}_h, \overline{\mathcal{F}}_h)\mid h\in K_0(\proj\Lambda), \overline{\mathcal{T}}_\theta\subseteq \overline{\mathcal{T}}_h, \overline{\mathcal{F}}_\theta\subseteq \overline{\mathcal{F}}_h\}, & \TF^{\mathrm{ss}}_{\mathbb{Z}}(\Lambda)=\bigcup\limits_{\scalemath{0.8}{\theta\in \mathbb{R}^n}}\TF^{\mathrm{ss}}_{\mathbb{Z}}(\theta).
		\end{array}\]
	\end{definition}
	\begin{remark}
		Let $\theta$ be an arbitrary vector. Then, the following assertions hold.
		\begin{itemize}
			\item[$(1)$] $(\overline{\mathcal{T}}_0, \overline{\mathcal{F}}_0)= (\mod\Lambda, \mod\Lambda)\in \TF^{\mathrm{ss}}_{\mathbb{Z}}(\theta)$.
			\item[$(2)$] For all $\gamma\in\overline{[\theta]_{\TF}}$, we have $\TF^{\mathrm{ss}}_{\mathbb{Z}}(\gamma)\subseteq\TF^{\mathrm{ss}}_{\mathbb{Z}}(\theta)$.
			\item[$(3)$] $\TF^{\mathrm{ss}}_{\mathbb{Z}}(\Lambda)$ is a partially ordered set ordered by inclusion.
		\end{itemize}
		Moreover, for a g-vector $g$, $|\TF^{\mathrm{ss}}_{\mathbb{Z}}(g)|=1$ if and only if $g=0$.
	\end{remark}
	\begin{corollary}\label{868524213012}
		Let $g$ be a g-vector. Then, each subset of  $\TF^{\mathrm{ss}}_{\mathbb{Z}}(\Lambda)$ has a minimal element.
	\end{corollary}
	\begin{proof}
		Choose an arbitrary subset $\mathcal{S}$ of $\TF^{\mathrm{ss}}_{\mathbb{Z}}(\Lambda)$. If there is no minimal element in $\mathcal{S}$, then one can find a strictly descending chain of the form
		\[\ldots\subsetneq (\overline{\mathcal{T}}_{\theta_2},\overline{\mathcal{F}}_{\theta_2})\subsetneq (\overline{\mathcal{T}}_{\theta_1},\overline{\mathcal{F}}_{\theta_1})\subsetneq (\overline{\mathcal{T}}_{\theta_0},\overline{\mathcal{F}}_{\theta_0}).\]
		Hence, by Proposition \ref{206559753624}, we conclude that there exists the following strictly ascending chain of positive integers:
		\[\ldots>\dim_{\mathbb{R}}\langle[\theta_2]_{\TF}\rangle_{\mathbb{R}}>\dim_{\mathbb{R}}\langle[\theta_1]_{\TF}\rangle_{\mathbb{R}}>\dim_{\mathbb{R}}\langle[\theta_0]_{\TF}\rangle_{\mathbb{R}}.\]
		However, $K_0(\proj\Lambda)_{\mathbb{R}}$ is of dimension $n$. This leads to a contradiction.
	\end{proof}
	\section{Generic decompositions and cones}
	Generic decomposition plays a central role in our study concerning $\TF$-equivalence classes and semistable torsion pairs.
	In this section, we provide some observations regarding the cones and $\TF$-equivalence classes of g-vectors. Specifically, using semistable torsion pairs and morphism torsion pairs, we prove (the second part of) Theorem \ref{310913388628}\footnote{The first part was proved by Asai and Iyama \cite[Corollary 3.15]{AsIy24}.}. First, we recall the notion of generic decomposition and tameness of g-vectors\footnote{For an overview of the notion of generic decomposition, (non-expert) readers are referred to \cite[Section 2]{HY25b}.}.
	
	Let $g$ be a g-vector. Set $\Hom_\Lambda(g)$ for $\Hom_\Lambda(P^{g-}, P^{g+})$, where $g=[P^{g+}]-[P^{g-}]$, and $P^{g+}$ and $P^{g-}$ are finitely generated projective modules without any common non-zero direct summands.
	Based on \cite[Theorem 4.4]{DeFe15}, there exist finitely many g-vectors $g_1,g_2,\ldots,g_s$ such that
	\begin{itemize}
		\item[$(1)$] for a general\footnote{``\textit{a general element} of a variety $\mathcal{X}$'' means an arbitrary element of a dense open subset of $\mathcal{X}$.} morphism $a$ in $\Hom_{\Lambda}(g)$, there are general morphisms $a_i\in\Hom_{\Lambda}(g_i)$, $1\le i\le s$ such that $a=a_1\oplus a_2\oplus\ldots\oplus a_s$, and
		\item[$(2)$] for each $1\le i\le s$, a general morphism of $\Hom_{\Lambda}(g_i)$ is indecomposable.
	\end{itemize}
	In this case, we write $g=g_1\oplus g_2\oplus\ldots\oplus g_s$. If a general morphism of $\Hom_{\Lambda}(g)$ is indecomposable, $g$ is said to be \textit{generically indecomposable}. The set of all generically indecomposable direct summands of $g$ is denoted by $\ind(g)$. Additionally,  we set $\ind(\mathbb{N}g)$ for the set of all generically indecomposable direct summands of elements in the set $\{tg\mid t\in\mathbb{N}\}$, and $\add(g)$ for the set of all direct summands of direct sums of $g$.
	\begin{definition}
		Let $g$ be a g-vector. If $2g=g\oplus g$, then $g$ is called \textit{tame}. Otherwise, it is said to be \textit{wild}. Moreover, generically indecomposable tame g-vectors in $\ind(g)$ are denoted by $\tame(\ind(g))$.
	\end{definition}
	\begin{notation}
		Let $g$ be a g-vector. Then, we define
		\[
		\begin{array}{l l}
			D_g:=\{h\in K_0(\proj\Lambda)\mid\exists s\in\mathbb{N}; g+sh=g\oplus sh\}, & D_{\mathbb{N}g}:=\bigcup_{t\in\mathbb{N}}D_{tg}.
		\end{array}
		\]
	\end{notation}
	The following lemma is established by Asai and Iyama \cite[Corollary 4.5]{AsIy24}. For the reader's convenience, we include its proof here.
	\begin{lemma}\label{076626606506}
		Let $g$ and $h$ be g-vectors. Then, the following conditions are equivalent.
		\begin{itemize}
			\item[$(1)$] $h\in D_g$.
			\item[$(2)$] There exists $a\in\Hom_{\Lambda}(g)$ such that $\mathcal{T}_{a}\subseteq\overline{\mathcal{T}}_h$ and $\mathcal{F}_a\subseteq\overline{\mathcal{F}}_h$.
			\item[$(3)$] There exists $a\in\Hom_{\Lambda}(g)$ such that $\Coker a\in\overline{\mathcal{T}}_h$ and $\Ker\nu a\in\overline{\mathcal{F}}_h$.
		\end{itemize}
		In this case, $\overline{[h]_{\TF}}\cap K_0(\proj\Lambda)\subseteq D_g$.
	\end{lemma}
	\begin{proof}
		By definition, the second and the third conditions are equivalent. Moreover, by \cite[Proposition 2.12]{HY25b}, the first condition holds if and only if there exists $s\in\mathbb{N}$, $a\in\Hom_{\Lambda}(g)$ and $b\in\Hom_{\Lambda}(sh)$ such that $e(a,b)=e(b,a)=0$. On the other hand, by Proposition \ref{341408431998}, the second condition is equivalent to the existence of $s\in\mathbb{N}$, $a\in\Hom_{\Lambda}(g)$ and $b\in\Hom_{\Lambda}(sh)$ such that $\mathcal{T}_a\subseteq\overline{\mathcal{T}}_b$ and $\mathcal{F}_a\subseteq\overline{\mathcal{F}}_b$. Therefore, the equivalence of conditions $(1)$ and $(2)$ follows from \cite[Proposition 3.11]{AsIy24}.
	\end{proof}
	\begin{proposition}\label{144840550539}
		Let $g$, $g'$, $h$ and $h'$ be g-vectors. Then the following statements hold.
		\begin{itemize}
			\item[$(1)$] $[g]_{\TF}=[g']_{\TF}$ implies that $D_{\mathbb{N}g}=D_{\mathbb{N}g'}$.
			\item[$(2)$] If both $h$ and $h'$ belong to the same $\TF$-equivalence class and $D_{g}$, then $th-h'\in D_g$ for some $t\in\mathbb{N}$.
		\end{itemize}
	\end{proposition}
	\begin{proof}
		The first assertion follows directly from definition. We prove the second one.
		By Lemma \ref{076626606506}, there is $a\in\Hom_{\Lambda}(g)$ such that $\Coker a\in\overline{\mathcal{T}}_h=\overline{\mathcal{T}}_{h'}$ and $\Ker\nu a\in\overline{\mathcal{F}}_h=\overline{\mathcal{F}}_{h'}$. Therefore, by Lemma \ref{284590845114}, there exists $t\in\mathbb{N}$ such that $\Coker a\in\overline{\mathcal{T}}_{th-h'}$ and $\Ker\nu a\in\overline{\mathcal{F}}_{th-h'}$. Thus, again, by Lemma \ref{076626606506}, $th-h'\in D_{g}$.
	\end{proof}
	
	\subsection{Cones}
	\begin{definition}
		Let $g$ be a g-vector. Then, the \textit{(open) cone of} $g$ is defined to be
		\[
		\Cone^{\circ}\{\ind(\mathbb{N}g)\}:=\bigcup_{t\in\mathbb{N}}\Cone^{\circ}\{\ind(tg)\}\subseteq\Cone\{\ind(\mathbb{N}g)\}^{\circ}.
		\]
	\end{definition}
	In \cite[Theorem 3.14]{AsIy24}, it is established that there are strong connections among these cones, $\TF$-equivalence classes and generic decompositions of g-vectors. This observation led Asai and Iyama to pose \cite[Conjecture 1.2]{AsIy24}, which is a main subject of our paper. This lemma relates the cone of a g-vector to its $\TF$-equivalence class, and is used repeatedly throughout the remainder of the paper.
	\begin{lemma}\label{206430023172} 
		Let $g$ be a g-vector and $g=g_1\oplus g_2\oplus \ldots \oplus g_s$ its generic decomposition. Then
		\[\begin{array}{l l l l l}
			\scalemath{.9}{\overline{\mathcal{T}}_{g}=\bigcap\limits_{i=1}^{s}\overline{\mathcal{T}}_{g_i}},&\scalemath{.9}{\overline{\mathcal{F}}_{g}=\bigcap\limits_{i=1}^{s}\overline{\mathcal{F}}_{g_i}},&\scalemath{.9}{\mathcal{T}_{g}=\bigcup\limits_{i=1}^{s}\mathcal{T}_{g_i}},&\scalemath{.9}{\mathcal{F}_{g}=\bigcup\limits_{i=1}^{s}\mathcal{F}_{g_i}},&\scalemath{.9}{\mathcal{W}_{g}=\bigcap\limits_{i=1}^{s}\mathcal{W}_{g_i}}.
		\end{array}\]
		Moreover, $\Cone^{\circ}\{\ind(\mathbb{N}g)\}\subseteq [g]_{\TF}$.
	\end{lemma}
	\begin{lemma}\label{806138874129}
		Consider a g-vector $g$. Then, for any $t,t'\in\mathbb{N}$,
		\[\begin{array}{l l}
			\Cone^{\circ}\{\ind(tg)\}\subseteq\Cone^{\circ}\{\ind(t'tg)\}, & \Cone\{\ind(tg)\}\subseteq\Cone\{\ind(t'tg)\}.
		\end{array}\]
		Therefore, for all $t\in\mathbb{N}$,
		\[\begin{array}{l l}
			\Cone\{\ind(\mathbb{N}tg)\}= \Cone\{\ind(\mathbb{N}g)\}, & 
			\Cone^{\circ}\{\ind(\mathbb{N}tg)\}= \Cone^{\circ}\{\ind(\mathbb{N}g)\}.
		\end{array}\]
	\end{lemma}
	\begin{proof}
		Assume that $g=g_1\oplus g_2\oplus\ldots\oplus g_s$ is the generic decomposition of $g$. Fix a g-vector $h\in\Cone\{\ind(g)\}$. Then, for some non-negative integers $u_i$, $0\le i \le s$, we have
		\[u_0 h=u_1 g_1\oplus\ldots\oplus u_s g_s.\]
		Therefore, for an arbitrary positive integer $t'$, $t'u_0 h=t'u_1 g_1\oplus\ldots\oplus t'u_s g_s$. Thus, $h\in\Cone\{\ind(t'g)\}$. Similarly, we can show that $\Cone^{\circ}\{\ind(g)\}$ is included in $\Cone^{\circ}\{\ind(t'g)\}$. Now, replace $g$ with $tg$ and complete the proof of the first part. Moreover, the second part follows immediately, from the definition and the first part.
	\end{proof}
	\begin{theorem}\label{758672546350}
		Consider a g-vector $g$. Then,
		\[\begin{array}{ll}
			\Cone\{\ind(\mathbb{N}g)\}= \bigcup\limits_{t\in\mathbb{N}}\Cone\{\ind(tg)\}, &
			\Cone\{\ind(\mathbb{N}g)\}^{\circ}= \Cone^{\circ}\{\ind(\mathbb{N}g)\}.
		\end{array}\]
		Moreover, $\partial\Cone\{\ind(\mathbb{N}g)\}\subseteq \bigcup\limits_{t\in\mathbb{N}} \partial\Cone\{\ind(tg)\}$.
	\end{theorem}
	\begin{proof}
		Obviously, the right-hand sides of the equations are included in the left-hand sides.
		For $t_i\in\mathbb{N}$, $1\le i\le s$, choose $h_i\in\ind(t_ig)$. Then  by Lemma \ref{806138874129}, there exists $t\in\mathbb{N}$ such that $t_i|t$ and $h_i\in\Cone\{\ind(tg)\}$, $1\le i\le s$. Therefore,
		\[\Cone\{h_i\mid 1\le i\le s\}\subseteq\Cone\{\ind(tg)\}.\]
		Thus, $\Cone\{\ind(\mathbb{N}g)\}\subseteq \bigcup\limits_{t\in\mathbb{N}}\Cone\{\ind(tg)\}$.
		
		Now, let $\theta$ be an arbitrary vector in $\Cone\{\ind(\mathbb{N}g)\}^{\circ}$. Set $\dim_{\mathbb{R}}\Cone\{\ind(\mathbb{N}g)\}=s$. Since $\Cone\{\ind(\mathbb{N}g)\}$ is a convex cone, we conclude that there exist linearly independent vectors $\theta_1$, $\theta_2$, $\ldots$, and $\theta_s$ in $\Cone\{\ind(\mathbb{N}g)\}$ such that $\theta\in\Cone^{\circ}\{\theta_i\mid 1\le i\le s\}$.
		By the first equality and Lemma \ref{806138874129}, it is easy to see that there exists $t\in\mathbb{N}$ such that for all $1\le i\le s$, $\theta_i\in\Cone\{\ind(tg)\}$ and $\dim_{\mathbb{R}}\Cone\{\ind(tg)\}=s$. This implies that $\Cone^{\circ}\{\theta_i\mid 1\le i\le s\}\subseteq\Cone^{\circ}\{\ind(tg)\}$.
		Therefore, $\theta$ must be included in $\Cone^{\circ}\{\ind(\mathbb{N}g)\}$.
		
		The last part follows immediately from the first part.
	\end{proof}
	\begin{lemma}\label{072929369060}
		Consider g-vectors $g$ and $h$. If there is a g-vector in
		\[\Cone^{\circ}\{\ind(\mathbb{N}g)\}\cap \Cone^{\circ}\{\ind(\mathbb{N}h)\},\]
		then $\Cone^{\circ}\{\ind(\mathbb{N}g)\}= \Cone^{\circ}\{\ind(\mathbb{N}h)\}$. In this case, there exist $t,s\in\mathbb{N}$ such that
		\[\begin{array}{l l}
			\add(\mathbb{N}tg)\subseteq \add(\mathbb{N}h), &
			\add(\mathbb{N}sh)\subseteq \add(\mathbb{N}g).
		\end{array}\]
	\end{lemma}
	\begin{proof}
		Based on Lemma \ref{806138874129}, without loss of generality, we may assume that there is a g-vector in $\Cone^{\circ}\{\ind(g)\}\cap \Cone^{\circ}\{\ind(h)\}$. Consider the generic decompositions $g=g_1\oplus\ldots\oplus g_s$ and $h=h_1\oplus\ldots\oplus h_t$. Then, there exist $c_i,d_i\in\mathbb{N}$, $1\le i \le s$ such that $c_1g_1\oplus\ldots\oplus c_sg_s=d_1h_1\oplus\ldots\oplus d_th_t$. Hence, for each $1\le i \le s$, $c_i g_i\in\add(\mathbb{N}h)$, and so $\Cone\{\ind(g)\}\subseteq\Cone\{\ind(\mathbb{N}h)\}$. Similarly, we can show that $\Cone\{\ind(h)\}\subseteq\Cone\{\ind(\mathbb{N}g)\}$.
	\end{proof}
	The following result refines the second assertion of \cite[Theorem 3.14]{AsIy24} (or Lemma \ref{206430023172}). Asai and Iyama established that for any g-vector $g$, we have $\Cone^{\circ}\{\ind(\mathbb{N}g)\}\subseteq [g]_{\TF}$.
	\begin{corollary}\label{342694140028}
		Let $g$, $h_1$ and $h_2$ be g-vectors. If $h_1\oplus h_2\in[g]_{\TF}^{\circ}$, then $\Cone^{\circ}\{h_1, h_2\}\subseteq [g]_{\TF}^{\circ}$. Therefore, $g\in[g]_{\TF}^{\circ}$ if and only if
		\[\Cone^{\circ}\{\ind(\mathbb{N}g)\}\subseteq  [g]^{\circ}_{\TF}.\]
	\end{corollary}
	\begin{proof}
		By Lemma \ref{206430023172}, we know that $h_1\oplus h_2$ and $th_1\oplus sh_2$ are $\TF$-equivalent, for all $t,s\in\mathbb{N}$. Hence, $t/(t+s)h_1+s/(t+s)h_2\in[g]_{\TF}$. Thus, $(h_1,h_2)\subseteq[g]_{\TF}$. On the other hand, by Lemma \ref{075288509646}, $\theta\in (h_1,h_2)\cap\partial[g]_{\TF}$ implies that $(\theta, h_2)\cap\overline{[g]_{\TF}}=\emptyset$. Therefore, $(h_1,h_2)\cap\partial[g]_{\TF}=\emptyset$. Thus, $(h_1,h_2)\subseteq [g]_{\TF}^{\circ}$. Hence, the assertions follow.
	\end{proof}
	\begin{theorem}\label{310913388628}
		For any two arbitrary g-vectors $g$ and $h$, if $\ind(g)=\ind(h)$, then they are $\TF$-equivalent. The converse holds if both are tame.
	\end{theorem}
	\begin{proof}
		If $\ind(g)=\ind(h)$, then by Lemma \ref{206430023172},
		\[[g]_{\TF}\supseteq\Cone^{\circ}\{\ind(g)\}=\Cone^{\circ}\{\ind(h)\}\subseteq [h]_{\TF}.\]
		So we have $[g]_{\TF}\cap [h]_{\TF}\neq\emptyset$. Hence, $[g]_{\TF}=[h]_{\TF}$.
		
		Conversely, let $g$ and $h$ be tame and assume that $[g]_{\TF}=[h]_{\TF}$. Since $g$ is tame, by Lemma \ref{076626606506}, we have $[g]_{\TF}\cap K_0(\proj\Lambda)\subseteq D_{h}$. Applying Lemma \ref{076626606506} again, one can find $b\in\Hom_{\Lambda}(h)$ such that $\mathcal{F}_b\subseteq\overline{\mathcal{F}}_g= \overline{\mathcal{F}}_h$ and $\mathcal{T}_b\subseteq\overline{\mathcal{T}}_g= \overline{\mathcal{T}}_h$. Hence, by Lemma \ref{284590845114}, for some $t\in\mathbb{N}$, $\mathcal{F}_b\subseteq\overline{\mathcal{F}}_{tg-h}$ and $\mathcal{T}_b\subseteq\overline{\mathcal{T}}_{tg-h}$. Therefore, $tg-h\in D_{h}$, and consequently, $s(tg-h)\oplus sh=stg$ for some $s\in\mathbb{N}$. This implies that $\ind(h)\subseteq\ind(g)$. Similarly, we can show $\ind(g)\subseteq\ind(h)$.
	\end{proof}
	\begin{remark}\label{050599086025}
		Let $g$ be a g-vector. Then, there exist finitely many proper subsets $H_1$, $H_2$, $\ldots$, and $H_s$ of $\ind(g)$ such that
		\[\begin{array}{cc}
			\partial \Cone\{\ind(g)\}= \bigcup\limits_{i=0}^{s}\Cone^{\circ}\{H_i\}, &
			\Cone^{\circ}\{H_i\}\cap\Cone^{\circ}\{H_j\}=\emptyset,\forall 1\le i\neq j\le s,
		\end{array}\]
		where we set $H_0=\emptyset$ and $\Cone^{\circ}\{H_0\}:=\{0\}$.
		
		Consider an arbitrary non-zero vector $\theta$ in $\partial\Cone\{\ind(\mathbb{N}g)\}$. Then  by Theorem \ref{758672546350}, there exists $t\in\mathbb{N}$ such that $\theta\in \partial\Cone\{\ind(tg)\}$. Therefore, $\theta\in\Cone^{\circ}\{H\}$, for some non-empty $H\subsetneq\ind(tg)$. Moreover, since $\langle H\rangle_{\mathbb{R}}\subseteq\langle \ind(\mathbb{N}g)\rangle_{\mathbb{R}}$, if the intersection $\Cone^{\circ}\{H\}\cap\Cone^{\circ}\{\ind(\mathbb{N}g)\}$ is non-empty, then it contains a g-vector. Thus, by Lemma \ref{072929369060}, we can conclude that
		\[\Cone^{\circ}\{H\}\subseteq \partial\Cone\{\ind(\mathbb{N}g)\}.\]
	\end{remark}
	\begin{lemma}\label{781520404735}
		Let $g$ be a g-vector which is not generically indecomposable. Then, $\partial\Cone\{\ind(g)\}\cap[g]_{\TF}\neq\emptyset$ implies that $\tame(\ind(\mathbb{N}g))\neq\emptyset$.
	\end{lemma}
	\begin{proof}
		Consider a g-vector $h$ in $\partial\Cone\{\ind(g)\}\cap[g]_{\TF}$. Then  by the notation of Remark \ref{050599086025}, $h$ belongs to $\Cone^{\circ}\{H_j\}$, for some $1\le j\le s$. Consequently, by Lemma \ref{076626606506}, for $h'\in\ind(g)\setminus H_j$, we have $h\in D_{h'}$. Therefore, $\overline{[g]_{\TF}}\subseteq D_{h'}$. On the other hand, by Lemma \ref{206430023172}, $\ind(g)\subseteq\overline{[g]_{\TF}}$. Hence, $h'\in D_{h'}$ and so $th'$ is tame, for some $t\in\mathbb{N}$.
	\end{proof}
	\subsection{Reduced {{g}}-vectors}
	\begin{lemma}\label{399829457548}
		Let $g$ be a g-vector.
		Then there exists $H\subseteq\ind(g)$ such that for all $h\in H$, $h\notin \Cone\{\ind(\mathbb{N}(H\setminus\{h\}))\}$ and
		\[\Cone^{\circ}\{\ind(\mathbb{N}H)\}=\Cone^{\circ}\{\ind(\mathbb{N}g)\}.\]
	\end{lemma}
	\begin{proof}
		Assume that $H_0:=\ind(g)=\{g_1,g_2,\ldots, g_s\}$. If $g_1\in \Cone\{\ind(\mathbb{N}(H_0\setminus\{g_1\}))\}$, then $g\in \Cone^{\circ}\{\ind(\mathbb{N}(H_0\setminus\{g_1\}))\}\cap\Cone^{\circ}\{\ind(\mathbb{N}g)\}$. Hence, by Lemma \ref{072929369060}, removing $g_1$ from $H_0$ does not affect the cone. In this case, we set $H_1=H_0\setminus\{g_1\}$; otherwise, $H_1=H_0$. Next, if $g_2\in \Cone\{\ind(\mathbb{N}(H_1\setminus\{g_2\}))\}$, then similarly to the previous step, it follows that removing $g_2$ from $H_1$ does not affect the cone. In this case, we set $H_2=H_1\setminus\{g_2\}$; otherwise, $H_2=H_1$. Applying this procedure successively to all elements of $\ind(g)$, we can find $H=H_s\subseteq\ind(g)$ as required.
	\end{proof}
	\begin{definition}\label{841845155042}
		Let $g$ be a non-zero g-vector. Then, we say $g$ is \textit{reduced}, provided that for any $h\in \ind(g)$, $h\notin \Cone\{\ind(\mathbb{N}(g-h))\}$.
	\end{definition}
		Any generically indecomposable g-vector is reduced. Moreover, if $g$ satisfies the ray condition (see \cite[Definition 2.20]{HY25b}) and there are no duplicate generically indecomposable g-vectors in its generic decomposition, then $g$ is reduced.
	\begin{remark}
		Let $g$ be a g-vector. Then  by Lemma \ref{399829457548}, there exists a reduced version of $g$, that is a direct summand of $g$ and belongs to $\Cone^{\circ}\{\ind(g)\}$.
	\end{remark}
	\begin{namedthm}{Lemma-definition}\label{090299082299}
		Consider a g-vector $g$.
		Let $g^{(r)}$ be a reduced g-vector in $\Cone^{\circ}\{\ind(\mathbb{N}g)\}$. Then, $\Cone^{\circ}\{\ind(\mathbb{N}g^{(r)})\}=\Cone^{\circ}\{\ind(\mathbb{N}g)\}$. We call $g^{(r)}$ a \textit{reduced version} of $g$.
	\end{namedthm}
	\begin{proof}
		It is a direct consequence of Lemmas \ref{072929369060} and \ref{399829457548}.
	\end{proof}
	\begin{corollary}
		Let $g$ be a g-vector. Then, $g\in[g]_{\TF}^{\circ}$ if and only if a reduced version of $g$ is in $[g]_{\TF}^{\circ}$.
	\end{corollary}
	\begin{proof}
		Let $g^{(r)}$ be a reduced version of $g$. Then  by Lemma \ref{090299082299},
		\[\Cone^{\circ}\{\ind(\mathbb{N}g)\}=\Cone^{\circ}\{\ind(\mathbb{N}g^{(r)})\}.\]
		Thus, the assertion follows from Corollary \ref{342694140028}.
	\end{proof}
	
	\begin{proposition}\label{700686515031}
		Let $g$ be a reduced g-vector. Then, $\ind(g)$ is linearly independent.
	\end{proposition}
	\begin{proof}
		Let $g=g_1\oplus g_2\oplus\ldots\oplus g_s$ be the generic decomposition of $g$. Consider an equation 
		\[a_1 g_1+a_2 g_2+\ldots+a_s g_s=b_1 g_1+b_2 g_2+\ldots+b_s g_s\]
		where $a_i, b_i\in\mathbb{Z}^{\ge 0}$ and if $a_i\neq 0$, then $b_i=0$, for all $1\le i\le s$. We can regard the above equation as
		\[a_1 g_1\oplus a_2 g_2\oplus \ldots\oplus a_s g_s=b_1 g_1\oplus b_2 g_2\oplus\ldots\oplus b_s g_s.\]
		If there are $j$ and $l$ such that $a_j\neq 0$ and $b_l\neq 0$, then $a_j g_j\in\add(\mathbb{N}(g-g_j))$. The contradiction implies that all coefficients must be zero. 
	\end{proof}
	
	\begin{proposition}\label{661790071494}
		Let $g$ be a g-vector. Then, there exists a reduced tame g-vector $f\in\add(\mathbb{N}g)$ such that $\ind(f)=\tame(\ind(\mathbb{N}g))$. Assume that $f$ is a direct summand of $g$. Then, for each $t\in\mathbb{N}$, we can write 
		\begin{equation}\label{eq876140490630}
			tg=f_t\oplus h_t,
		\end{equation}
		where
		\begin{itemize}
			\item $f^{\oplus t}$ is a direct summand of $f_t$,
			\item $f_t\in\Cone^{\circ}\{\ind(f)\}$, and
			\item $h_t$ does not admit any direct summands contained in $\Cone\{\ind(f)\}$.
		\end{itemize}
		Moreover, if $f'_t\oplus h'_t\in[g]_{\TF}$, where $0\neq h'_t$ and $f'_t$ are direct summands of $h_t$ and $f_t$, respectively, then $h'_t=h_t$.
	\end{proposition}
	\begin{proof}
		From \cite[Theorem 2.30 and Proposition 2.27]{HY25b}, we know that $\tame(\ind(\mathbb{N}g))$ is linearly independent. Therefore, $\tame(\ind(\mathbb{N}g))$ has only finitely many members. Thus, there exists a tame g-vector $f'\in\add(\mathbb{N}g)$ such that
		$\ind(f')=\tame(\ind(\mathbb{N}g))$. By removing duplicate generically indecomposable g-vectors from the generic decomposition of $f'$, we obtain a reduced tame g-vector $f$.
		
		For $t\in\mathbb{N}$, consider the decomposition $tg=f_t\oplus h_t$, where $f_t$ is the maximum direct summand of $tg$ contained in $\Cone\{\ind(f)\}$. Since $f^{\oplus t}$ is a direct summand of $tg$, we conclude that $f^{\oplus t}$ is a direct summand of $f_t$ and so $f_t\in\Cone^{\circ}\{\ind(f)\}$.
		
		To prove the last part, consider the decomposition $h_t=h'_t\oplus h''_t$. If $h''_t\neq 0$,
		by Lemma \ref{076626606506}, we can show that $f'_t\oplus h'_t\in D_{h''_t}$, and so $\overline{[g]_{\TF}}\subseteq D_{h''_t}$. Therefore, $h''_t$ belongs to $D_{h''_t}$. Thus, $sh''_t$ is tame, for some $s\in\mathbb{N}$. Hence, $h''_t$ lies in $\Cone\{\ind(f)\}$. This leads to a contradiction.
	\end{proof}
	\begin{lemma}\label{833777223956}
		Let $g$ be a reduced g-vector and $|\ind(g)|\ge 2$. Then, for any non-empty subset $H\subsetneq\ind(g)$, 
		$\add(\mathbb{N}H)\subseteq\partial [g]_{\TF}$.
		Moreover, if $g\in[g]_{\TF}^{\circ}$, then
		\[\ind(g)\subseteq\partial[g]_{\TF}\setminus[g]_{\TF}.\]
	\end{lemma}
	\begin{proof}
		It follows from Lemma \ref{206430023172} that $\add(\mathbb{N}H)\subseteq\overline{[g]_{\TF}}$. Thus, it is enough to show that $\add(\mathbb{N}H)\cap [g]_{\TF}^{\circ}=\emptyset$.
		Let $g=g_1\oplus g_2\oplus\ldots\oplus g_{s\ge 2}$ be the generic decomposition of $g$. Take $g_j\in\ind(g)\setminus H$. Then  by Lemma \ref{076626606506}, $\add(\mathbb{N}H)\subseteq D_{g_j}$, and therefore, $\add(\mathbb{N}H)\cap[g]_{\TF}^{\circ}\neq\emptyset$ implies that $\overline{[g]_{\TF}}\subseteq D_{g_j}$. Consider $h\in \add(\mathbb{N}H)\cap[g]_{\TF}^{\circ}$. Then, since $g_j\in\overline{[g]_{\TF}}$ and $h\in [g]_{\TF}^{\circ}$, by Lemma \ref{434203923932}, $f=th-g_j\in [g]_{\TF}$, for some $t\in\mathbb{N}$. Thus,
		\[a(th-g_j)\oplus ag_j=ath\]
		for some $a\in\mathbb{N}$. Therefore, $ag_j\in\add(\mathbb{N}H)\subseteq\Cone\{\mathbb{N}H\}$. Since $g$ is reduced, this leads to a contradiction. 
		
		By the first assertion, we know that $\ind(g)\subseteq\partial[g]_{\TF}$. Let $g_1\in[g]_{\TF}$. Since $g_1\in D_{g-g_1}$, by Lemma \ref{076626606506}, $g\in D_{g-g_1}$.
		Therefore, by Proposition \ref{144840550539}, we have $tg_1-g\in D_{g-g_1}$, for some $t\in\mathbb{N}$. Thus, there exists $t'\in\mathbb{N}$ such that
		\[t'(tg_1-g)+(g-g_1)=t'(tg_1-g)\oplus(g-g_1).\]
		By Lemma \ref{206430023172}, this implies that all g-vectors in $\Cone^{\circ}\{tg_1-g,g-g_1\}$ are $\TF$-equivalent. On the other hand, by Lemma \ref{075288509646}, we know that \[(g_1,tg_1-g)\cap\overline{[g]_{\TF}}=\emptyset.\]
		However, both $g$ and $g_1$ belong to $\Cone^{\circ}\{tg_1-g,g-g_1\}\cap[g]_{\TF}$, which implies that $(g_1,tg_1-g)\subseteq [g]_{\TF}$. This leads to a contradiction. Therefore, we conclude that $g_1\notin[g]_{\TF}$. Similarly, we can prove that $g_i\notin [g]_{\TF}$, for all $1\le i \le s$.
	\end{proof}
	\section{Main results}
	This section presents our results regarding the cones of g-vectors and $\TF$-equivalence classes. We begin by demonstrating that Conjecture \ref{362256186221} holds for tame g-vectors. Additionally, we establish that the cones of g-vectors are rational and simplicial. Subsequently, we show that for any g-vector $g$, there exists $t\in\mathbb{N}$ such that $tg$ satisfies the ray condition. As a consequence, we prove that the interior of the $\TF$-equivalence class and the open cone of a given g-vector coincide if and only if they are of the same dimension.  Finally, we provide a necessary and sufficient condition for Conjecture \ref{362256186221} to hold.
	\subsection{Cones of g-vectors and $\TF$-equivalence classes}
	\begin{theorem}\label{376103330122}
		Let $g$ be a tame g-vector. Then
		\[[g]_{\TF}=\Cone^{\circ}\{\ind(g)\}.\]
	\end{theorem}
	\begin{proof}
		Consider an arbitrary g-vector $h\in [g]_{\TF}$. Using Lemma \ref{076626606506}, one can show that $h\in D_{\mathbb{N}h}$. Therefore, there exists $s\in\mathbb{N}$ such that $sh$ is tame. Now, by Theorem \ref{310913388628}, we conclude that $\ind(g)=\ind(sh)$. Hence, $sh\in\Cone^{\circ}\{\ind(g)\}$ and we have $[g]_{\TF}\cap K_{0}(\proj\Lambda)_{\mathbb{Q}}\subseteq\Cone^{\circ}\{\ind(g)\}$. Consequently, it follows that $[g]_{\TF}\subseteq\Cone\{\ind(g)\}$. 
		
		Since the facets of $\Cone\{\ind(g)\}$ are generated by rational vectors, rational points form dense subsets in each of them. Thus, it is sufficient to show that 
		\[\partial\Cone\{\ind(g)\}\cap[g]_{\TF}\cap K_0(\proj\Lambda)=\emptyset.\]
		Consider a g-vector $h$ in $\partial\Cone\{\ind(g)\}$. Then  by Remark \ref{050599086025}, there exists $H\subsetneq\ind(g)$ such that $h\in\Cone^{\circ}\{H\}$. Hence, it follows from Theorem \ref{310913388628} that $\ind(th)=H$, for some $t\in\mathbb{N}$. Therefore $th$ is tame. Since $\ind(th)\neq\ind(g)$, again, it follows from Theorem \ref{310913388628} that $th\notin [g]_{\TF}$.
	\end{proof}
	Unlike the tame case, if $g$ does not satisfy the ray condition, the inequality \[\Cone\{\ind(\mathbb{N}g)\}\neq\Cone\{\ind(g)\}\]
	might happen (see \cite[Example 5.9]{AsIy24}). In the following theorem, we show that there exists a finite set of rational generators for $\Cone\{\ind(\mathbb{N}g)\}$. Therefore, it is rational and polyhedral.
	\begin{theorem}\label{076273627938}
		Let $g$ be a g-vector. Then, there exists $t\in\mathbb{N}$ such that
		\[\Cone\{\ind(\mathbb{N}g)\}=\Cone\{\ind(tg)\}.\]
	\end{theorem}
	\begin{proof}
		First, note that if there is a tame g-vector in $\Cone^{\circ}\{\ind(\mathbb{N}g)\}$, then the assertion follows from Theorem \ref{376103330122} and Lemma \ref{072929369060}. Thus, we can assume that there is no tame g-vector in $\Cone^{\circ}\{\ind(\mathbb{N}g)\}$.
		Consider the assumptions of Proposition \ref{661790071494}. Without loss of generality, let $f$ be a direct summand of $g$. Then, $\Cone\{\ind(f)\}\subseteq \partial\Cone\{\ind(\mathbb{N}g)\}$.
		
		By Lemma \ref{781520404735}, if $f=0$, then for all $t\in\mathbb{N}$,
		\[\partial\Cone\{\ind(tg)\}\cap \Cone^{\circ}\{\ind(\mathbb{N}g)\}=\emptyset.\]
		Choose $t\in\mathbb{N}$ with $\dim_{\mathbb{R}}\langle\ind(tg)\rangle_{\mathbb{R}}= \dim_{\mathbb{R}} \langle\ind(\mathbb{N}g)\rangle_{\mathbb{R}}$. Then, it is easy to see that $\Cone^{\circ}\{\ind(\mathbb{N}g)\}=\Cone^{\circ}\{\ind(tg)\}$.
		
		Now, assume that $f\neq 0$.
		Choose a vector $\theta$ in $\partial\Cone\{\ind(\mathbb{N}g)\}$ such that $g$ belongs to $(f, \theta)$. By Remark \ref{050599086025}, there exist $t\in\mathbb{N}$ and $H\subsetneq\ind(tg)$ such that $\theta\in\Cone^{\circ}\{H\}\subseteq \partial\Cone\{\ind(\mathbb{N}g)\}$. Therefore, there exists a g-vector $h_1\in\Cone^{\circ}\{H\}$ such that $g$ is contained in $\Cone^{\circ}\{f,\ind(h_1)\}$.
		We may assume that $g=f\oplus h_1$, and $h_1$ does not admit any direct summands in $\Cone\{\ind(f)\}$. Since $h_t$ in the decomposition \eqref{eq876140490630} is a direct summand of $th_1$, by Lemma \ref{434203923932} and Corollary \ref{342694140028}, we conclude that $h_t$ must be in $\partial\Cone\{(\ind(\mathbb{N}g))\}$, for all $t\in\mathbb{N}$.
		
		Consider the finite set $\{f_i\mid i\in I\}$ consisting of all proper direct summands of $f$.
		Then  by Proposition \ref{661790071494}, Remark \ref{050599086025} and Theorem \ref{758672546350}, it is straightforward to show that
		\begin{equation}\label{Qu2bf5hg5PSp}
			\scalemath{0.9}{\partial\Cone\{\ind(tg)\}\cap\Cone^{\circ}\{\ind(\mathbb{N}g)\}\neq\emptyset\iff f_i\oplus h_t\in \partial\Cone\{\ind(tg)\}\cap\Cone^{\circ}\{\ind(\mathbb{N}g)\}}
		\end{equation}
		for some $i\in I$. In the following, for each $i\in I$, we find $t_i\in\mathbb{N}$ such that $f_i\oplus h_{t_i}\in\partial\Cone\{\ind(\mathbb{N}g)\}$:
		
		Assume that $f_i\oplus h_1\in \Cone^{\circ}\{\ind(\mathbb{N}g)\}$, for some $i\in I$. Then, similar to the above, we can find $t'_i\in\mathbb{N}$, and direct summands $f'$ of $f$ and $h'_{t'_i}$ of $h_{t'_i}$ such that $f'\oplus h'_{t'_i}\in \partial\Cone\{\ind(\mathbb{N}g)\}$ and
		$f_i\oplus h_1\in\Cone^{\circ}\{(f-f_i),\ind (f'\oplus h'_{t'_i})\}$ (see Figure \ref{385950501174}).
		Therefore, $f_i\in\Cone\{\ind(f'\oplus h'_{t'_i})\}$. Hence,
		\[f_i\oplus h'_{t'_i}\in\Cone\{\ind(f'\oplus h'_{t'_i})\}\subseteq\partial\Cone\{\ind(\mathbb{N}g)\}.\]
		Since $h_1\in\Cone\{\ind(f\oplus h'_{t'_i})\}$, by Proposition \ref{661790071494}, we know that there exists $t_i,c\in\mathbb{N}$ such that $h_{t_i}$ is a direct summand of $ch'_{t'_i}$. Thus,
		\[f_i\oplus h_{t_i}\in\Cone\{\ind(f_i\oplus h'_{t'_i})\}\subseteq\partial\Cone\{\ind(\mathbb{N}g)\}.\] 
		
		Note that since for $a\in\mathbb{N}$, $f_i\oplus h_{at}$ is a direct summand of $f_i\oplus h_{t}$, if $f_i\oplus h_{t}$ belongs to $\partial\Cone\{\ind(\mathbb{N}g)\}$, then so is $f_i\oplus h_{at}$. 
		Hence, by \eqref{Qu2bf5hg5PSp}, for some
		$\prod_{i\in I} t_i\mid t$, we have $\partial\Cone\{\ind(tg)\}\cap\Cone^{\circ}\{\ind(\mathbb{N}g)\}=\emptyset$ and $\langle\ind(\mathbb{N}g)\rangle_{\mathbb{R}}= \langle\ind(tg)\rangle_{\mathbb{R}}$. Therefore,
		\[\Cone\{\ind(\mathbb{N}g)\}= \Cone\{\ind(tg)\}.\]
		\begin{figure}
			\centering
			\textbf{}\par\medskip
			\[\begin{tikzpicture}[scale=0.35, rotate around z=240]
				\coordinate (A) at (0,0);
				\coordinate (B) at (10,0);
				\coordinate (C) at (3,3);
				\coordinate (E) at (5,5);
				\coordinate (F) at (6,6);
				\coordinate (O) at (5,20);
				\node at (O) [below] {$0$};
				\node at (A) [left] {$\scalemath{0.5}{f-f_i}$};
				\node at (B) [below] {$\scalemath{0.5}{f_i}$};
				\node at (C) [right] {$\scalemath{0.5}{h_1}$};
				\node at (E) [right] {$\scalemath{0.5}{h_{t'_i}}$};
				\node at (F) [below] {$\scalemath{0.5}{h_{t_i}}$};
				\filldraw[black] (7,3) circle (3pt);
				\draw (O) -- (A);
				\draw (O) -- (B);
				\draw (O) -- (C);
				\draw (O) -- (E);
				\draw (O) -- (F);
				\draw (O) -- (A);
				\draw (A) -- (B) node [midway,left] {$\scalemath{0.5}{f}$};
				\filldraw[black] (5,0) circle (3pt);
				\draw[thick,dotted,red] (5,0) -- (C);
				\filldraw[black] (4.45,0.8) circle (3pt);
				\node at (4.1,0.8) {$\scalemath{0.5}{g}$};
				\draw (A) -- (F);
				\draw (B) -- (C);
				\node at (5.6,1.29) {$\scalemath{0.5}{f_i\oplus h_1}$};
				\filldraw[black] (5.07,2.15) circle (3pt);
				\draw (B) -- (E) node [midway, right=2.5mm] {$\scalemath{0.5}{f_{i}\oplus h_{t'_i}}$};
				\draw (B) -- (F);
				\draw[thick,dotted,blue] (A) -- (7,3);
			\end{tikzpicture}\]
			\caption{The idea behind the proof of Theorem \ref{076273627938}}\label{385950501174}
		\end{figure}
	\end{proof}
	\begin{corollary}\label{355706191811}
		Let $g$ be a g-vector. Then, $tg$ satisfies the ray condition, for some $t\in\mathbb{N}$. Therefore, the cone of $g$ is simplicial.
	\end{corollary}
	\begin{proof}
		By Theorem \ref{076273627938},
		without loss of generality, we can assume that $\Cone\{\ind(\mathbb{N}g)\}=\Cone\{\ind(g)\}$.
		We show that there exists a g-vector in $\Cone^{\circ}\{\ind(g)\}$ satisfies the ray condition. Since $\Cone\{\ind(g)\}$ is a rational polyhedral cone, there are finitely many g-vectors $\{g_1,g_2,\ldots,g_s\}$ such that $\{\mathbb{R}^{>0}g_1,\mathbb{R}^{>0}g_2,\ldots,\mathbb{R}^{>0}g_s\}$ is the complete set of pairwise distinct rays of $\Cone\{\ind(g)\}$. Assume that for some $t\in\mathbb{N}$ and $1\le i\le s$, we have $tg_i=h_1\oplus h_2$. Then  by Lemma \ref{206430023172}, $g\in\Cone\{h_1,h_2\}\subseteq\Cone\{\ind(g)\}$. Since $\mathbb{R}^{>0}g_i$ is a ray, we conclude that $h_1,h_2\in\mathbb{R}^{>0}g$. Therefore, in this case, it follows from Lemma \ref{076626606506} that $g_i\in D_{\mathbb{N}g_i}$. Thus, $t_ig_i$ is tame and $|\ind(t_ig_i)|=1$, for some $t_i\in\mathbb{N}$. Hence, there exists $t'\in\mathbb{N}$ such that for all $1\le i\le s$,  $|\ind(t'g_i)|=1$. Thus,
		\[g':=t'\sum_{i=1}^{s}g_i=\bigoplus_{i=1}^{s}t'g_i\]
		satisfies the ray condition. Moreover, by \cite[Theorem 4]{HY25b}, the set $\{t'g_1, t'g_2,\ldots, t'g_s\}$ is linearly independent. Hence, $\Cone\{\ind(g')\}=\Cone\{\ind(\mathbb{N}g)\}$ is simplicial.
	\end{proof}
	\begin{corollary}\label{276955124928}
		Let $g$ be a g-vector. Then, it admits a reduced version $g'$ satisfying the ray condition and $|\ind(g')|=\dim_{\mathbb{R}}\langle\ind(\mathbb{N}g)\rangle_{\mathbb{R}}$.
	\end{corollary}
	
	\begin{lemma}\label{988807280630}
		Let $g$ be a g-vector. Then
		\[\Cone\{\ind(\mathbb{N}g)\}\cap [g]^{\circ}_{\TF}\subseteq\Cone^{\circ}\{\ind(\mathbb{N}g)\}\subseteq [g]_{\TF}.\]
		Especially, if $g\in[g]^{\circ}_{\TF}$, then
		\[\Cone\{\ind(\mathbb{N}g)\}\cap [g]^{\circ}_{\TF}=\Cone^{\circ}\{\ind(\mathbb{N}g)\}.\]
	\end{lemma}
	\begin{proof}
		First, note that Lemma \ref{206430023172} along with Theorem \ref{758672546350} imply that $\Cone^{\circ}\{\ind(\mathbb{N}g)\}\subseteq [g]_{\TF}$.
		Moreover, it is straightforward to check that if $\dim_{\mathbb{R}}\langle\ind(\mathbb{N}g)\rangle_{\mathbb{R}}=1$, then
		\[\Cone\{\ind(\mathbb{N}g)\}\cap [g]^{\circ}_{\TF}\subseteq\Cone^{\circ}\{\ind(\mathbb{N}g)\},\]
		and by Corollary \ref{342694140028}, equality holds, if $g$ belongs to $[g]_{\TF}^{\circ}$. Moreover, if there exists a tame g-vector in $[g]_{\TF}$, then the assertions follow from Lemma \ref{072929369060} and Theorem \ref{376103330122}. Otherwise, assume that $\dim_{\mathbb{R}}\langle\ind(\mathbb{N}g)\rangle_{\mathbb{R}} \ge 2$ and there is no tame g-vector in $[g]_{\TF}$. We break down the proof into two cases.\\
		\\
		\textbf{First case:}
		Let $\tame(\ind(\mathbb{N}g))=\emptyset$. Then  by Lemma \ref{781520404735}, we can see that for all $t\in\mathbb{N}$,
		$\partial\Cone\{\ind(tg)\}\cap[g]_{\TF}=\emptyset$.
		Therefore, by Theorem \ref{758672546350}, $\partial\Cone\{\ind(\mathbb{N}g)\}\cap[g]_{\TF}= \emptyset$. Thus,
		\[\Cone\{\ind(\mathbb{N}g)\}\cap[g]_{\TF}= \Cone^\circ\{\ind(\mathbb{N}g)\}.\]
		\\
		\textbf{Second case:}
		Let $\tame(\ind(\mathbb{N}g))\neq\emptyset$. Fix the notation of Proposition \ref{661790071494}. Based on Theorem \ref{076273627938}, without loss of generality, we can assume that
		\begin{itemize}
			\item $\Cone\{\ind(\mathbb{N}g)\}=\Cone\{\ind(g)\}$,
			\item $f$ and $h_1$ are direct summands of $g$, and
			\item $f, h_1\in\partial\Cone\{\ind(g)\}$.
		\end{itemize}
		Now, let
		\begin{equation}\label{eq459033724587}
			\partial\Cone\{\ind(g)\}\cap[g]_{\TF}^{\circ}\neq \emptyset.
		\end{equation}
		Then  by Proposition \ref{661790071494}, $f'\oplus h_1$ is included in the above intersection, for some direct summand $f'$ of $f$. Therefore, since $f$ and $f'\oplus h_1$ are integer vectors, it is easy to see that there is a g-vector $h'\in [g]_{\TF}\setminus\Cone\{\ind(g)\}$ such that $f'\oplus h_1\in\Cone^{\circ}\{f,h'\}$. Additionally, by Lemma \ref{076626606506}, $f+sh'=f\oplus sh'$, for some $s\in\mathbb{N}$.
		However, by Lemma \ref{072929369060},
		\[\Cone\{\ind(g)\}= \Cone\{\ind(\mathbb{N}(f\oplus sh'))\}.\]
		Since $h'\notin\Cone\{\ind(g)\}$, this leads to a contradiction. Thus, \eqref{eq459033724587} never holds.
		
		Moreover, The second part is a direct consequence of the first part and Corollary \ref{342694140028}.
	\end{proof}
	We now characterize precisely when the interior of the $\TF$-equivalence class coincides with the corresponding open cone of a given g-vector.
	\begin{theorem}\label{962650431975}
		For any g-vector $g$, the following statements are equivalent.
		\begin{itemize}
			\item[$(1)$]
			$\dim_{\mathbb{R}}\langle [g]_{\TF} \rangle_{\mathbb{R}}= \dim_{\mathbb{R}}\langle\ind(\mathbb{N}g)\rangle_{\mathbb{R}}$.
			\item[$(2)$]
			$[g]_{\TF}^{\circ}=\Cone^{\circ}\{\ind(\mathbb{N}g)\}$.
		\end{itemize}
	\end{theorem}
	\begin{proof}
		In the case that $\dim_{\mathbb{R}}\langle [g]_{\TF} \rangle_{\mathbb{R}}= \dim_{\mathbb{R}}\langle\ind(\mathbb{N}g)\rangle_{\mathbb{R}}=1$, the assertion is clear.
		Let $\dim_{\mathbb{R}}\langle [g]_{\TF} \rangle_{\mathbb{R}}= \dim_{\mathbb{R}}\langle\ind(\mathbb{N}g)\rangle_{\mathbb{R}}\ge 2$. Then, $\Cone^{\circ}\{\ind(\mathbb{N}g)\}$ is an open subset of $[g]_{\TF}$. Thus, $g$ belongs to $[g]_{\TF}^{\circ}$. Therefore, by Lemma \ref{988807280630}, we have $\Cone\{\ind(\mathbb{N}g)\}\cap [g]^{\circ}_{\TF}=\Cone^{\circ}\{\ind(\mathbb{N}g)\}$. Hence, $\partial\Cone\{\ind(\mathbb{N}g)\}\cap[g]_{\TF}^{\circ}=\emptyset$.
		Thus, $(1)$ implies $[g]_{\TF}^{\circ}\subseteq\Cone^{\circ}\{\ind(\mathbb{N}g)\}$. Consequently, by Corollary \ref{342694140028}, $(1)$ implies $(2)$. 
	\end{proof}
	
	\subsection{Equivalence of the two Conjectures}
	\begin{proposition}\label{061272712412}
		Let $g$ be a g-vector. If Conjecture \ref{362256186221} holds true, then
		$\TF_{\mathbb{Z}}^{\mathrm{ss}}(g)<\infty$.
	\end{proposition}
	\begin{proof}
		It follows from Theorem \ref{076273627938} that $[g]_{\TF}^{\circ}=\Cone^{\circ}\{\ind(tg)\}$, for some $t\in\mathbb{N}$. Therefore, the assertion follows from Remark \ref{050599086025} and Lemma \ref{206430023172}.
	\end{proof}
	\begin{proposition}\label{117791826723}
		Consider a g-vector $g$ with $|\TF^{\mathrm{ss}}_{\mathbb{Z}}(g)|= 2$. Then, $\dim_{\mathbb{R}}\langle\ind(\mathbb{N}g)\rangle_{\mathbb{R}}=1$. Moreover, if Conjecture \ref{404319523362}(2) holds, then $\dim\langle[g]_{\TF}\rangle_{\mathbb{R}}=1$.
	\end{proposition}
	\begin{proof}
		Assume that $\dim_{\mathbb{R}}\langle\ind(\mathbb{N}g)\rangle_{\mathbb{R}}\ge2$. Thus, there are linearly independent g-vectors $g_1$ and $g_2$ such that $tg=g_1\oplus g_2$, for some $t\in\mathbb{N}$. Consequently, by the assumption, $g_1$ and $g_2$ must be contained in $[g]_{\TF}$. Since $g_2\in D_{g_1}$, we also have $g\in D_{g_1}$. Therefore, $t_1g_1$ is tame, for some $t_1\in\mathbb{N}$. Similarly, $t_2g_2$ is tame, for some $t_2\in\mathbb{N}$. Hence, $t_1t_2g$ is tame. Therefore, by Theorem \ref{376103330122}, $[g]_{\TF}=\Cone^{\circ}\{\ind(tt_1t_2g)\}$. This implies that neither $g_1$ nor $g_2$ lies in $[g]_{\TF}$, which contradicts the assumption that $|\TF^{\mathrm{ss}}_{\mathbb{Z}}(g)|= 2$.
	\end{proof}
	
	\begin{theorem}\label{948850386317}
		Conjecture \ref{362256186221} and Conjecture \ref{404319523362} are equivalent.
	\end{theorem}
	\begin{proof}
		It is straightforward to check that if $[g]_{\TF}^{\circ}=\Cone^{\circ}\{\ind(\mathbb{N}g)\}$, then rational points are dense in $[g]_{\TF}$. Moreover, by Proposition \ref{061272712412}, we can see that Conjecture \ref{362256186221} implies Conjecture \ref{404319523362}(3). Additionally, Conjecture \ref{404319523362}(2) follows obviously from Conjecture \ref{362256186221}. Therefore, if Conjecture \ref{362256186221} holds, then so does Conjecture \ref{404319523362}.
		
		Conversely,
		assume that Conjecture \ref{404319523362} holds true.
		First, consider the case where $\dim_{\mathbb{R}}\langle [g]_{\TF} \rangle_{\mathbb{R}}=1$. In this case, the assertion follows directly from Theorem \ref{962650431975}.
		Next, using Proposition \ref{206559753624}, we prove the remaining cases by induction on $\dim_{\mathbb{R}}\langle [g]_{\TF} \rangle_{\mathbb{R}}$.
		
		Assume that $\dim_{\mathbb{R}}\langle [g]_{\TF} \rangle_{\mathbb{R}}\ge 2$.
		Then, according to Proposition \ref{117791826723} and Conjecture \ref{404319523362}(3), there are finitely many g-vectors $h^1,\ldots ,h^{s\ge 2}\in\partial [g]_{\TF}\setminus [g]_{\TF}$ such that $[h^i]_{\TF}\neq[h^j]_{\TF}$, for all $0\le i\neq j\le s$, and
		\[\overline{[g]_{\TF}}\cap K_0(\proj\Lambda)_{\mathbb{Q}}\subseteq [g=h^{0}]_{\TF}\cup(\bigcup_{i=1}^{s}[h^i]^{\circ}_{\TF}).\]
		
		From Proposition \ref{206559753624}, we know that $\dim_{\mathbb{R}}\langle [h^i]_{\TF} \rangle_{\mathbb{R}}\lneqq\dim_{\mathbb{R}}\langle [g]_{\TF} \rangle_{\mathbb{R}}$, for each $1\le i\le s$. Thus, by induction hypothesis, $[h^i]_{\TF}^{\circ}=\Cone^{\circ}\{\ind(\mathbb{N}h^i)\}$. Next, take a reduced g-vector $g'\in [g]^{\circ}_{\TF}$. Then  by Conjecture \ref{404319523362}(2) and Corollary \ref{276955124928}, we can assume that
		\[\dim_{\mathbb{R}}\langle\ind(g')\rangle_{\mathbb{R}}= \dim_{\mathbb{R}} \langle\ind(\mathbb{N}g')\rangle_{\mathbb{R}} \ge2.\]
		Note that by Lemma \ref{833777223956}, $\ind(g')\subseteq\partial [g]_{\TF}\setminus[g]_{\TF}$. Therefore, the set
		\[I':=\{0\le i \le s\mid\ind(g')\cap [h^i]^{\circ}_{\TF}\neq\emptyset\}\]
		does not contain zero.
		Hence, by Lemma \ref{072929369060}, there exists $a\in\mathbb{N}$ such that
		\[
		\ind(\mathbb{N}ag')\subseteq\bigcup\limits_{i\in I'}\ind(\mathbb{N}h^i).
		\]
		Therefore, $g'\in\Cone\{\bigcup_{i\in I'}\ind(\mathbb{N}h^i)\}$.
		On the other hand, by Lemma \ref{076626606506}, for all $i\neq j\in I'$, $h^i\in D_{\mathbb{N}h^j}$. Thus, there exists $t\in\mathbb{N}$ such that
		\[h^{I'}:=t\sum_{i\in I'}h^i=\bigoplus_{i\in I'}th^i.\]
		It is straightforward to check that $g'\in\Cone\{\ind(\mathbb{N}h^{I'})\}$.
		Hence, for each reduced g-vector $g'\in[g]_{\TF}^{\circ}$, there is a subset $I'$ of $\{1,\ldots,s\}$ such that $g'\in\Cone\{\ind(\mathbb{N}h^{I'})\}$. Since $\{1,\ldots,s\}$ has only finitely many subsets, the set $[g]_{\TF}^{\circ}\cap K_0(\proj\Lambda)_{\mathbb{Q}}$ can be covered by finitely many cones of dimension lower than or equal to $\dim_{\mathbb{R}}\langle [g]_{\TF} \rangle_{\mathbb{R}}$. Thus, Conjecture \ref{404319523362}(1) implies there exists a subset $I'\subseteq\{1,\ldots,s\}$ such that
		\[\dim_{\mathbb{R}}\langle \ind(\mathbb{N}h^{I'}) \rangle_{\mathbb{R}}=\dim_{\mathbb{R}}\langle [g]_{\TF} \rangle_{\mathbb{R}}.\]
		Therefore, $\Cone^{\circ}\{\ind(\mathbb{N}h^{I'})\}\cap [g]_{\TF}^{\circ}\neq\emptyset$. Since $\Cone\{\ind(\mathbb{N}h^{I'})\}\subseteq \overline{[g]_{\TF}}$, we conclude that $\Cone^{\circ}\{\ind(\mathbb{N}h^{I'})\}\subseteq [g]_{\TF}^{\circ}$.
		Thus, by Theorem \ref{962650431975},
		\[[g]_{\TF}^{\circ}=\Cone^{\circ}\{\ind(\mathbb{N}h^{I'})\}.\]
	\end{proof}
	\begin{corollary}\label{401237755092}
		Let $g$ be a g-vector. If $\dim_{\mathbb{R}}W_g=|\Lambda|-\dim_{\mathbb{R}} \langle\ind(\mathbb{N}g)\rangle_{\mathbb{R}}$, then $[g]_{\TF}^{\circ}=\Cone^{\circ}\{\ind(\mathbb{N}g)\}$. The converse also holds, if moreover, $W_{g}=\Ker\langle [g]_{\TF},-\rangle$.
		In particular, if $\dim_{\mathbb{R}}\langle\ind(\mathbb{N}g)\rangle_{\mathbb{R}}=|\Lambda|-1$, then \[[g]_{\TF}^{\circ}=\Cone^{\circ}\{\ind(\mathbb{N}g)\}.\]
	\end{corollary}
	\begin{proof}
		It is easy to see that $W_g\subseteq\ker\langle h,-\rangle$, for all $h\in[g]_{\TF}$. Thus,
		\begin{equation}\label{401237755090}
			\dim_{\mathbb{R}}W_g\le|\Lambda|-\dim_{\mathbb{R}}\langle[g]_{\TF}\rangle_{\mathbb{R}}\le |\Lambda|-\dim_{\mathbb{R}} \langle\ind(\mathbb{N}g)\rangle_{\mathbb{R}}.
		\end{equation}
		Therefore, the first and second assertions follows from Theorem \ref{962650431975}.
		
		By Proposition \ref{129173072599}, $W_g\neq\{0\}$. Thus, by \eqref{401237755090},
		\[\dim_{\mathbb{R}}W_{g}=\dim_{\mathbb{R}}\Ker \langle\ind(\mathbb{N}g) ,-\rangle=1.\]
		Therefore, the last assertion is a consequence of the first one.
	\end{proof}
	
	\section*{Acknowledgment}
	The authors thank Sota Asai for pointing out a gap in the proof of Theorem \ref{076273627938} in an earlier draft of this manuscript. The first author received an INNS Research Fellowship, and this work was partially funded by the Iran National Science Foundation (INSF), Project No. 4003197.
	
	{\small\bibliographystyle{alpha}
		\bibliography{all}}
\end{document}